\documentclass[11pt,a4paper,reqno]{amsart}
\usepackage{epsf,epsfig,amssymb,amsfonts,amsgen,amsmath,amstext,amsbsy,amsopn,amsthm,lineno}
\usepackage{color,comment,colortbl}
\usepackage{caption}
\usepackage{graphicx,tikz,enumerate}
\usetikzlibrary{arrows,shapes,positioning,decorations.markings}
\usepackage{hyperref}

\newtheorem{theorem}{Theorem}[section]

\newtheorem{proposition}[theorem]{Proposition}
\newtheorem{lemma}[theorem]{Lemma}
\newtheorem{corollary}[theorem]{Corollary}

\newtheorem{conjecture}[theorem]{Conjecture}

\newcommand{\Sym}{\mathrm{Sym}}
\newcommand{\Cay}{\mathrm{Cay}}
\newcommand{\Aut}{\mathrm{Aut}}
\newcommand{\Soc}{\mathrm{Soc}}
\newcommand{\I}{\mathcal{I}}
\newcommand{\N}{\mathbf{N}}

\newcommand{\Core}{\mathrm{Core}}
\newcommand{\CG}{\mathrm{CG}}
\newcommand{\GRR}{\mathrm{GRR}}
\newcommand\Cy{\mathrm{C}}
\newcommand\Quat{\mathrm{Q}}

\numberwithin{equation}{section}
\allowdisplaybreaks

\title{\textbf{Asymptotic enumeration of graphical regular representations}}

\author[Xia]{Binzhou Xia}
\address{(Xia) School of Mathematics and Statistics\\The University of Melbourne\\Parkville 3010\\VIC\\Australia}
\email{binzhoux@unimelb.edu.au}

\author[Zheng]{Shasha Zheng}
\address{(Zheng) School of Mathematics and Statistics\\The University of Melbourne\\Parkville 3010\\VIC\\Australia}
\email{zhesz@student.unimelb.edu.au}

\begin{document}

\begin{abstract}
We estimate the number of graphical regular representations (GRRs) of a given group with large enough order. As a consequence, we show that almost all finite Cayley graphs have full automorphism groups `as small as possible'. This confirms a conjecture of Babai-Godsil-Imrich-Lov\'{a}sz on the proportion of GRRs, as well as a conjecture of Xu on the proportion of normal Cayley graphs, among Cayley graphs of a given finite group.



\end{abstract}

\maketitle

\section{Introduction}

Let $R$ be a finite group and $S$ a subset of $R$. The \emph{Cayley digraph} of $R$ with \emph{connection set} $S$, denoted by $\Cay(R,S)$, is defined as the digraph with vertex set $R$ such that for $u,v\in R$, $(u,v)$ is an arc if and only if $vu^{-1}\in S$. Note that $1\in S$ is allowed in our definition of Cayley digraphs and that a Cayley digraph $\Cay(R,S)$ is a graph if and only if $S=S^{-1}$. For convenience, we identify a group $R$ with its right regular permutation representation, which is obviously a subgroup of $\Aut(\Cay(R,S))$. A digraph $\Gamma$ is said to be a \emph{digraphical regular representation} (DRR for short) of a group $R$ if its full automorphism group $\Aut(\Gamma)$ is isomorphic to $R$ and acts on the vertex set of $\Gamma$ as a regular permutation group. Thus DRRs of $R$ are precisely Cayley digraphs of $R$ with full automorphism group as small as $R$. If a DRR of $R$ is a graph, then it is called a \emph{graphical regular representation} (GRR for short) of $R$.

It is natural to ask which finite groups admit DRRs or GRRs. The question for GRRs was studied in a series of papers~\cite{Chao1964,Nowitz1968,Sabidussi1964,Watkins1971} before a complete answer was given by Godsil~\cite{Godsil1978} in 1978: Apart from two infinite families and thirteen small solvable groups (of order at most $32$), every finite group admits a GRR. Here one of the two infinite families consists of abelian groups of exponent greater than $2$ (see \cite{Chao1964,Sabidussi1964}). The other family is the so called generalized dicyclic groups, namely, non-abelian groups $R$ that have an abelian normal subgroup $A$ of index $2$ and an element $x\in R\setminus A$ of order $4$ such that $x^{-1}ax=a^{-1}$ for all $a\in A$ (see \cite{Nowitz1968,Watkins1971}). The situation for DRRs is simpler, as shown by Babai~\cite{Babai1980}: $\Cy_2^2$, $\Cy_2^3$, $\Cy_2^4$, $\Cy_3^2$ and $\Quat_8$ are the only five groups without DRRs.

It was conjectured by Babai and Godsil~\cite{Babai-Godsil1982,Godsil1981} that almost all finite Cayley digraphs are DRRs. This conjecture has been confirmed recently by Morris and Spiga~\cite{Morris-S2021}. In fact, they proved the following quantified version:

\begin{theorem}[Morris-Spiga]\label{Thm:MS}
Let $R$ be a group of order $r$, where $r$ is sufficiently large. The proportion of subsets $S$ of $R$ such that $\Cay(R,S)$ is a DRR of $R$ is at least $1-2^{-\frac{br^{0.499}}{4\log_2^3r}+2}$, where $b$ is an absolute constant.

\end{theorem}

Since abelian groups of exponent greater than $2$ and generalized dicyclic groups are the only infinite families without GRRs, a corresponding conjecture on the proportion of GRRs among Cayley graphs was made by Babai, Godsil, Imrich and Lov\'{a}sz, as stated in~\cite[Conjecture 2.1]{Babai-Godsil1982}:

\begin{conjecture}[Babai-Godsil-Imrich-Lov\'{a}sz]\label{BGIL}
Let $R$ be a group of order $r$ that is neither abelian of exponent greater than $2$ nor generalized dicyclic. The proportion of inverse-closed subsets $S$ of $R$ such that $\Cay(R,S)$ is a GRR of $R$ approaches $1$ as $r$ approaches infinity.
\end{conjecture}


The conjecture has been tackled by various authors with partial results and important progress, see~\cite{Babai-Godsil1982,Morris-MP2022,Spiga-eq2021} for instance. In this paper, we completely solve Conjecture~\ref{BGIL} in the affirmative:

\begin{theorem}\label{Thm:main}
Let $R$ be a group of order $r$ such that $R$ is neither abelian of exponent greater than $2$ nor generalized dicyclic, and let $P(R)$ be the proportion of inverse-closed subsets $S$ of $R$ such that $\Cay(R,S)$ is a GRR of $R$. There exists an absolute constant $c_0$ such that whenever $r\geq c_0$, the proportion $P(R)$ is at least $1-2^{-\frac{r^{0.499}}{8\log_2^3{r}}+\log_2^2{r}+3}$. In particular, $P(R)$ approaches $1$ as $r$ approaches infinity.
\end{theorem}

For a group $R$ that is neither abelian of exponent greater than $2$ nor generalized dicyclic, let $\CG(R)$ denote the set of Cayley graphs on $R$ up to isomorphism, and let $\GRR(R)$ denote the set of GRRs on $R$ up to isomorphism. Elements of $\CG(R)$ are called \emph{unlabeled} Cayley graphs on $R$. In contrast, to emphasize that a Cayley graph $\Cay(R,S)$ on $R$ is determined by the set $S$, we sometimes call it a \emph{labeled} Cayley graph.
The following result, which can be derived easily from Theorem~\ref{Thm:main}, shows that almost all unlabeled Cayley graphs on such groups $R$ are GRRs.

\begin{corollary}\label{unlab}
Let $R$ be a group of order $r$ such that $R$ is neither abelian of exponent greater than $2$ nor generalized dicyclic. There exists an absolute constant $c_0$ such that whenever $r\geq c_{0}$, the ratio $|\GRR(R)|/|\CG(R)|$ is at least $1-2^{-\frac{r^{0.499}}{8\log_2^3{r}}+2\log_2^2{r}+3}$. In particular, $|\GRR(R)|/|\CG(R)|$ approaches $1$ as $r$ approaches infinity.
\end{corollary}

A Cayley digraph $\Cay(R,S)$ is said to be \emph{normal} if $R$ is normal in $\Aut(\Cay(R,S))$. Studying the normality of Cayley digraphs is a helpful approach to understand the symmetry of Cayley digraphs. In \cite{Xu1998}, Xu conjectured that almost all Cayley (di)graphs are normal. The digraph part of this conjecture has been confirmed by Morris and Spiga in \cite{Morris-S2021} as an immediate corollary of Theorem~\ref{Thm:MS}, while the graph part of this conjecture, stated below, has been open.

\begin{conjecture}[Xu]\label{Xu}
Let $R$ be a group of order $r$ such that $R\ncong \Quat_8\times \Cy_2^m$ for any nonnegative integer $m$. The proportion of inverse-closed subsets $S$ of $R$ such that $\Cay(R,S)$ is a normal Cayley graph of $R$ approaches $1$ as $r$ approaches infinity.
\end{conjecture}

We remark that the excluded groups $\Quat_8\times \Cy_2^m$ are the Hamiltonian $2$-groups, which are proved in~\cite[Lemma~5]{Wang1998} to be the groups other than $\Cy_4\times \Cy_2$ that admit no normal Cayley graphs.
It is also worth noting that almost all Cayley graphs on finite groups $R$ without GRRs have full automorphism groups `as small as possible', see Dobson-Spiga-Verret~\cite[Theorem 1.5]{Dobson2016} for abelian groups and Morris-Spiga-Verret~\cite[Theorem 1.4]{Morris-SV2015} for generalized dicyclic groups.
Combining these results with our Theorem~\ref{Thm:main}, we confirm Conjecture~\ref{Xu} by the quantified version in the next corollary.
In fact, the equivalence between Conjectures~\ref{BGIL} and~\ref{Xu} was already known to Spiga~\cite{Spiga-eq2021}.

\begin{corollary}\label{normality}
Let $R$ be a group of order $r$ such that $R\ncong \Quat_8\times \Cy_2^m$ for any nonnegative integer $m$. For large enough $r$, the proportion of inverse-closed subsets $S$ of $R$ such that $\Cay(R,S)$ is a normal Cayley graph of $R$ is at least $1-2^{-\frac{r^{0.499}}{8\log_2^3{r}}+\log_2^2{r}+3}$.
\end{corollary}

The rest of this paper is devoted to the proof of Theorem~\ref{Thm:main} and Corollaries~\ref{unlab} and~\ref{normality}.
The strategy we apply to prove Theorem~\ref{Thm:main} is initiated in~\cite{Babai-Godsil1982}, followed by~\cite{Dobson2016,Morris-SV2015,Spiga-FDR2018,Spiga-GFR2020}, and further developed in~\cite{Morris-S2021}.
We also make use of the substantial progress on enumeration of GRRs achieved very recently by Spiga~\cite{Spiga-suborbits2021+} and Morris-Moscatiello-Spiga~\cite{Morris-MP2022}.
Our main contribution to the solution of Conjecture~\ref{BGIL} lies in the results in Section~\ref{Sec3}, which eventually leads us to Theorem~\ref{Thm:main} from several important pieces of work by various authors.

\section{Preliminaries}\label{Sec2}

For a subset $X$ of a group, denote by $\I(X)$ the set of elements of $X$ of order at most $2$ and let
\begin{equation}\label{eq:cr}
\mathbf{c}(X)=\frac{|X|+|\mathcal{I}(X)|}{2}.
\end{equation}
It is clear (see~\cite[Lemma 2.2]{Spiga-eq2021} for instance) that, if $X$ is inverse-closed, then the number of inverse-closed subsets of $X$ is $2^{\mathbf{c}(X)}$.

\begin{lemma}\emph{(\cite[Theorem 4.1]{Edmonds2009})}\label{lem:inv-ea2}
A finite group $R$ such that $|\I(R)|/|R|>3/4$ is an elementary abelian $2$-group.
\end{lemma}

From the identity
\begin{equation}\label{eq:parity}
\sum_{k=0}^m(-1)^k\binom{m}{k}=0
\end{equation}
we see that, for a given nonempty set, the number of subsets of odd size equals the number of subsets of even order. In particular,
\begin{equation}\label{eq:BoundBinom}
\binom{m}{k}\leq2^{m-1}\ \text{ for all }\, k\in\{0,1,\dots,m\}.
\end{equation}

\begin{lemma}\label{lem:nics-givensize}
Let $R$ be a group, and let $X$ be a nonempty inverse-closed subset of $R$. The number of inverse-closed subsets of $X$ with a given size is at most $2^{\mathbf{c}(X)-1}$.
\end{lemma}

\begin{proof}
Let $n=|X|$. For a nonnegative integer $k$, denote the number of inverse-closed $k$-subsets of $X$ by $n_k$.

Case 1:  $n$ is odd.

Let $A=\{S\subseteq X: S=S^{-1},\, |S| \text{ is odd}\}$, $B=\{S\subseteq X: S=S^{-1},\, |S| \text{ is even}\}$, and let $f :  A  \rightarrow B$
be the map such that $f(S)=X\setminus S$. It is clear that $f$ is a bijection. This together with the observation $|A\cup B|=2^{\mathbf{c}(X)}$ implies that $|A|=|B|=2^{\mathbf{c}(X)-1}$. Hence we have $n_k\leq2^{\mathbf{c}(X)-1}$ for every nonnegative integer $k$.

Case 2:  $n$ is even and $\I(X)\neq\emptyset$.

Take an element $x$ in $\I(X)$. Then $X\setminus \{x\}$ is an inverse-closed subset of size $n-1$, which is odd. By the conclusion of Case 1, we know that the number of inverse-closed $k$-subsets of $X\setminus \{x\}$ is at most
\begin{equation}\label{eq:noalpha}
2^{\mathbf{c}(X\setminus \{x\})-1}=2^{\frac{(|X\setminus \{x\}|+|\mathcal{I}(X\setminus \{x\})|)}{2}-1}=2^{\frac{(|X|-1+|\mathcal{I}(X)|-1)}{2}-1}=2^{\frac{|X|+|\mathcal{I}(X)|}{2}-2}=2^{\mathbf{c}(X)-2}.
\end{equation}
For a nonnegative integer $\ell$, let $C_{\ell}=\{S\subseteq X: S=S^{-1},\, x\notin S \text{ and } |S|=\ell\}$, $D_{\ell}=\{S\subseteq X: S=S^{-1},\, x \in S \text{ and } |S|=\ell\}$, and let $g_{\ell} :  C_{\ell}  \rightarrow D_{\ell+1}$ be the map such that $g_{\ell}(S)=S\cup\{x\}$. It is clear that $g_{\ell}$ is a bijection between $C_{\ell}$ and $D_{\ell+1}$. Hence for each positive integer $k$, the number of inverse-closed $k$-subsets of $X$ that contains $x$ is exactly the number of inverse-closed $(k-1)$-subsets of $X\setminus \{x\}$, which is at most $2^{\mathbf{c}(X)-2}$ by~\eqref{eq:noalpha}. Combining this with~\eqref{eq:noalpha}, we obtain $n_k\leq2^{\mathbf{c}(X)-2}+2^{\mathbf{c}(X)-2}=2^{\mathbf{c}(X)-1}$ for every nonnegative integer $k$.

Case 3:  $n$ is even and $\I(X)=\emptyset$.

Suppose $n=2m$ and write $X=\{x_1,x_2,\ldots, x_{2m}\}$ with $x_i^{-1}=x_{i+m}$ for $i\in\{1,\dots,m\}$. Note that an inverse-closed subset $S$ of $X$ has even order and is determined by $S\cap\{x_1,x_2,\ldots, x_{m}\}$, whose size is exactly $|S|/2$. Hence for a nonnegative even integer $k$, we have
\[
n_k\leq\binom{m}{k/2}\leq2^{m-1}=2^{\mathbf{c}(X)-1},
\]
where the second inequality is from~\eqref{eq:BoundBinom}.
Since $n_k=0$ when $k$ is odd, it follows that $n_k\leq2^{\mathbf{c}(X)-1}$ for every nonnegative integer $k$.
\end{proof}

The following result, proved in \cite[Section~4]{Spiga-suborbits2021+}, is useful in estimating the number of Cayley graphs on $R$ whose full automorphism group contains a given overgroup of $R$.


\begin{lemma}[Spiga]\label{lem:GTS}
Let $G$ be a transitive group properly containing a regular subgroup $R$ on $R$, where the group $R$, identified with its right regular permutation representation, is neither abelian of exponent greater than $2$ nor generalized dicyclic. Let $G_1$ be the stabilizer of $1\in R$ in $G$, and let $\iota$ be the permutation on $R$ sending every element to its inverse. The number of $\langle G_1, \iota\rangle$-orbits on $R$ is at most $\mathbf{c}(R)-|R|/96$. In particular, the number of (labeled) Cayley graphs $\Gamma$ on $R$ with $G\leq \Aut(\Gamma)$ is at most $2^{\mathbf{c}(R)-\frac{|R|}{96}}$.
\end{lemma}

The next lemma is from~\cite[Theorem 1.6]{Morris-MP2022}, whose proof involves applying results from~\cite{Spiga-eq2021}.

\begin{lemma}[Morris-Moscatiello-Spiga]\label{lem:STEP2}
Let $R$ be a finite group of order $r$ that is neither abelian of exponent greater than $2$ nor generalized dicyclic, and let $N$ be a nontrivial proper normal subgroup of $R$. The number of inverse-closed subsets $S$ of $R$ such that there exists a non-identity $g\in\N_{\Aut(\Cay(R,S))}(N)$ that fixes the vertex $1$ and stabilizes each $N$-orbit is at most $2^{\mathbf{c}(R)-\frac{r}{192|N|}+\log_2^2{r}+3}$.
\end{lemma}

Let $G$ be a group that acts transitively on a set $\Omega$ with $|\Omega|>1$. A nonempty subset $B$ of $\Omega$ is called a \emph{block} of $G$ if, for each $g\in G$, either $B^g=B$ or $B^g\cap B=\emptyset$. In this case, we call $\{B^g:g \in G\}$ a \emph{block system} of $G$. It is clear that the group $G$ has $\Omega$ and the singletons as blocks, which are viewed as the \emph{trivial} ones. The group $G$ is said to be \emph{primitive} if it only has trivial blocks. It is easy to see that $G$ is primitive if and only if a point stabilizer in $G$ is a maximal subgroup of $G$. Note that there are eight O'Nan-Scott types of the finite primitive permutation groups: HA, HS, HC, SD, CD, TW, AS, PA, as divided in~\cite{Baddeley2003}.

\begin{lemma}\emph{(\cite[Lemmas 2.4 and 6.1]{Morris-SV2015})}\label{lem:typeex}
A primitive permutation group with an abelian point stabilizer or a generalized dicyclic point stabilizer is of type $\mathrm{HA}$.
\end{lemma}

Let $\Gamma$ be a digraph on vertex set $V$. A partition of $V$ into sets $C_1,C_2,\ldots,C_t$ is said to be \emph{equitable} if for a vertex $v\in C_i$, the number of outneighbours of $v$ in $C_j$ depends only on the choice of $C_i$ and $C_j$, that is, the number of outneighbours of any $v\in C_i$ in $C_j$ is a constant $b_{i,j}$. The following observation shows that if $G$ is a subgroup of $\Aut(\Gamma)$, then the partition formed by the $G$-orbits on $V$ is  equitable.

\begin{lemma}\label{lem:orbit-equitable}
Let $\Gamma$ be a digraph, and let $G\leq\Aut(\Gamma)$. Given $G$-orbits $B$ and $C$ on the vertex set of $\Gamma$, vertices in $B$ are adjacent to the same number of vertices in $C$.
\end{lemma}

\begin{proof}
For a vertex $v$ of $\Gamma$, denote the set of outneighbours of $v$ by $\mathrm{N}^{+}(v)$. Let $u,w\in B$. There is an element $\varphi\in G\leq\Aut(\Gamma)$ such that $u^{\varphi}=w$ and $(\mathrm{N}^{+}(u))^{\varphi}=\mathrm{N}^{+}(w)$. Note that $C^{\varphi}=C$. Hence
\[
|\mathrm{N}^{+}(u)\cap C|=|(\mathrm{N}^{+}(u)\cap C)^{\varphi}|=|(\mathrm{N}^{+}(u))^{\varphi}\cap C^{\varphi}|=|\mathrm{N}^{+}(w)\cap C|,
\]
which means $u$ and $w$ have the same number of outneighbours in $C$.
\end{proof}

We close this section with a somewhat technical lemma.

\begin{lemma}\label{lem:normalorbit}
Let $R<G\leq\Sym(R)$ such that $R$, identified with its right regular permutation representation, is a maximal subgroup of $G$, and let $L$ be a normal subgroup of $G$. The $L$-orbit on $R$ containing $1$ is a subgroup of $R$. If it is further a normal subgroup of $R$, then either $L\leq R$ or every $L$-orbit is stabilized by $G_1$.
\end{lemma}

\begin{proof}
Since $L$ is normal in $G$, the $L$-orbits on $R$ form a block system of $G$. Let $Q=1^L$ be the $L$-orbit containing $1$. Since $R$ acts regularly by right multiplication, we have $Q\leq R$. Hence the block system formed by the $L$-orbits coincides with the set of right $Q$-cosets of $R$.
Assume that $Q$ is a normal subgroup of $R$.

Let $M$ be the largest subgroup of $G$ stabilizing every $L$-orbit. Hence $M\trianglelefteq G$ and so $MR$ is a group. Now $R\leq MR\leq G$. Since $R$ is maximal in $G$, we conclude that either $R=MR$ or $MR=G$. If $R=MR$, then $L\leq M\leq R$, as desired. (In this case, we actually have $L=M$, since $M\leq G_L=G_1L$ and $R\cap G_1L=(R\cap G_1)L=L$.)
Next assume $MR=G$.
Take an arbitrary $g\in G_1$. It follows that $Q^g=1^{Lg}=1^{gL}=1^L=Q$. Since $g\in G=MR$, we may write $g=hx$ with $h\in M$ and $x\in R$, which implies
$Q^g=Q^{hx}=Qx$. This together with $Q^g=Q$ leads to $x\in Q$. Observing that the normal subgroup $Q$ of $R$ stabilizes every right $Q$-coset of $R$, we have $Q\leq M$ and so $g=hx\in MQ=M$. Therefore, $G_1\leq M$, which means that every $L$-orbit is stabilized by $G_1$.
\end{proof}

\section{Subsets evenly intersecting cosets}\label{Sec3}

Let $R$ be a group with a subgroup $Q$, and let $\Delta$ be a union of some right $Q$-cosets in $R$, say, $\Lambda_1,\Lambda_2,\dots,\Lambda_b$ (so that $|\Delta|=b|Q|$). We say that a subset $S$ of $R$ \emph{intersects $\Delta$ evenly} if
\[
|S\cap\Lambda_1|=|S\cap\Lambda_2|=\dots=|S\cap\Lambda_b|.
\]
We are interested in the case when $\Delta$ comes from an orbit of some group acting on the set $R/Q$ of right $Q$-cosets in $R$.
For the group actions in the two subsections below, we estimate the number of such $S$.

\subsection{Cosets from a double coset}

For a subgroup $Q$ of $R$, under the action of $Q$ on $R/Q$ by right multiplication, an orbit of $Q$ is a double coset of $Q$ in $R$.
This subsection is devoted to the proof of the following proposition, which estimates the number of inverse-closed subsets in a group that intersect evenly with all the double cosets of a given subgroup. Recall the notation
\[
Q\backslash R/Q=\{QxQ: x\in R\}
\]
for the set of double $Q$-cosets in $R$.

\begin{proposition}\label{prop:ns-nonnorm}
Let $R$ be a group, let $Q$ be a non-normal subgroup of $R$ and let
\[
\mathcal{L}=\{S\subseteq R: S=S^{-1},\, S \text{ intersects $\Delta$ evenly for each } \Delta\in Q\backslash R/Q\}.
\]
We have $|\mathcal{L}|\leq 2^{\mathbf{c}(R)-\frac{1}{8}|R\,{:}\,Q|}$, where $\mathbf{c}(R)=(|R|+|\mathcal{I}(R)|)/2$ as in~\eqref{eq:cr}.
\end{proposition}

In what follows we are going to provide a series of lemmas, which will lead to a proof of Proposition~\ref{prop:ns-nonnorm} at the end of this subsection.



\begin{lemma}\label{Claimfor2k}
Let $R$ be a group and $QxQ$ an inverse-closed double coset of a subgroup $Q$ in $R$ with $|QxQ|=b|Q|$, where $b\geq2$. For any ordering $(\Phi_{1},\Phi_{2},\ldots,\Phi_{b})$ of the $b$ right $Q$-cosets in $QxQ$, we have
\[
\bigcup_{i=k+1}^{2k}\bigcup_{j=k+1}^{b}(\Phi_{i}\cap\Phi_{j}^{-1})\neq\emptyset\ \textup{ for all }\,k\in\{1,\ldots,\lfloor b/2\rfloor\}.
\]
\end{lemma}

\begin{proof}
Since $QxQ$ is inverse-closed and $\{\Phi_{1},\Phi_{2},\ldots,\Phi_{b}\}$ forms a partition of $QxQ$, we have
\[
QxQ=\Phi_{1}^{-1}\cup\cdots\cup\Phi_{b}^{-1}.
\]
Suppose for a contradiction that $\cup_{i=k+1}^{2k}\cup_{j=k+1}^{b}(\Phi_{i}\cap\Phi_{j}^{-1})=\emptyset$ for some $k\in\{1,\ldots,\lfloor b/2\rfloor\}$. Hence for each $i\in\{k+1,\ldots,2k\}$, we deduce that $\cup_{j=k+1}^{b}(\Phi_{i}\cap\Phi_{j}^{-1})=\emptyset$ and so
\[
\Phi_i\subseteq QxQ\setminus (\Phi_{k+1}^{-1}\cup\cdots\cup\Phi_{b}^{-1})=\Phi_{1}^{-1}\cup\cdots\cup\Phi_{k}^{-1},
\]
which implies that
\[
\Phi_{k+1}\cup\cdots\cup\Phi_{2k}\subseteq\Phi_{1}^{-1}\cup\cdots\cup\Phi_{k}^{-1}.
\]
Combining this with
\[
|\Phi_{k+1}\cup\cdots\cup\Phi_{2k}|=k|Q|=|\Phi_{1}^{-1}\cup\cdots\cup\Phi_{k}^{-1}|,
\]
we obtain
\[
\Phi_{k+1}\cup\cdots\cup\Phi_{2k}=\Phi_{1}^{-1}\cup\cdots\cup\Phi_{k}^{-1}.
\]
Since $\Phi_{k+1},\ldots,\Phi_{2k}$ are right $Q$-cosets, it follows that
\[
Q(\Phi_{1}^{-1}\cup\cdots\cup\Phi_{k}^{-1})=Q(\Phi_{k+1}\cup\cdots\cup\Phi_{2k})=\Phi_{k+1}\cup\cdots\cup\Phi_{2k}=\Phi_{1}^{-1}\cup\cdots\cup\Phi_{k}^{-1}.
\]
Write $\Phi_{i}=Qy_i$ for each $i\in\{1,\ldots,k\}$.
Note that
\[
\{y_1^{-1},\ldots,y_k^{-1}\}\subseteq\Phi_{1}^{-1}\cup\cdots\cup\Phi_{k}^{-1}.
\]
Then we derive that
\begin{align*}
Qy_1^{-1}\cup\cdots\cup Qy_k^{-1}\subseteq Q(\Phi_{1}^{-1}\cup\cdots\cup\Phi_{k}^{-1})=\Phi_{1}^{-1}\cup\cdots\cup\Phi_{k}^{-1}=y_1^{-1}Q\cup \cdots\cup y_k^{-1}Q.
\end{align*}
As a consequence,
\[
QxQ=Qy_1^{-1}Q\cup\cdots\cup Qy_k^{-1}Q=(Qy_1^{-1}\cup \cdots\cup Qy_k^{-1})Q\subseteq y_1^{-1}Q\cup \cdots\cup y_k^{-1}Q.
\]
This implies that $b|Q|=|QxQ|\leq k|Q|$, contradicting to the condition that $2k\leq b$.
\end{proof}

\begin{lemma}\label{Claimforarrangement}
Let $R$ be a group and $QxQ$ an inverse-closed double coset of a subgroup $Q$ in $R$ with $|QxQ|=b|Q|$, where $b\geq2$.
There exists an ordering $(\Lambda_{1},\Lambda_{2},\ldots,\Lambda_{b})$ of the $b$ right $Q$-cosets in $QxQ$ such that
\[
\bigcup_{j=i}^{b}(\Lambda_{i}\cap\Lambda_{j}^{-1})\neq\emptyset\ \textup{ for all }\,i\in\{2,\ldots,\lfloor b/2\rfloor+1\}.
\]
\end{lemma}

\begin{proof}
Take any ordering $(\Phi_{1},\Phi_{2},\ldots,\Phi_{b})$ of the $b$ right $Q$-cosets in $QxQ$.  Starting from this ordering, we construct in the following a series of orderings inductively, which will finally lead to a desired ordering $(\Lambda_{1},\Lambda_{2},\ldots,\Lambda_{b})$.
It is known by Lemma~\ref{Claimfor2k} (taking $k=1$ there) that the inequality
\[
\bigcup_{j=i}^{b}(\Phi_{i}\cap\Phi_{j}^{-1})\neq\emptyset
\]
holds for $i=2$.

Suppose that, for some $\ell\in\{2,\ldots,\lfloor b/2\rfloor\}$, there is an ordering $(\Phi_{\ell,1},\Phi_{\ell,2},\ldots,\Phi_{\ell,b})$ of the $b$ right $Q$-cosets $\Phi_{1},\Phi_{2},\ldots,\Phi_{b}$ such that
\[
\bigcup_{j=i}^{b}(\Phi_{\ell,i}\cap\Phi_{\ell,j}^{-1})\neq\emptyset\ \text{ for all }\,i\in\{2,\ldots,\ell\}.
\]
(Note that this supposition holds for $\ell=2$ by letting $\Phi_{2,j}=\Phi_{j}$ for $j\in\{1,\ldots,b\}$.)
Applying Lemma~\ref{Claimfor2k} to the ordering $(\Phi_{\ell,1},\Phi_{\ell,2},\ldots,\Phi_{\ell,b})$, we have
\[
\bigcup_{i=\ell+1}^{2\ell}\bigcup_{j=\ell+1}^{b}(\Phi_{\ell,i}\cap\Phi_{\ell,j}^{-1})\neq\emptyset.
\]
Take the smallest $s\in\{\ell+1,\ldots,2\ell\}$ such that
\begin{equation}\label{smallests}
\bigcup_{j=\ell+1}^{b}(\Phi_{\ell,s}\cap\Phi_{\ell,j}^{-1})\neq\emptyset.
\end{equation}
If $s=\ell+1$, then let $(\Phi_{\ell+1,1},\Phi_{\ell+1,2},\ldots,\Phi_{\ell+1,b})=(\Phi_{\ell,1},\Phi_{\ell,2},\ldots,\Phi_{\ell,b})$. If $s>\ell+1$, then let
\begin{equation*}
\Phi_{\ell+1,i}=
\begin{cases}
\Phi_{\ell,i} &\text{ if }i\in\{1,2,\ldots,b\}\setminus \{\ell+1,s\}, \\
\Phi_{\ell,s} &\text{ if }i=\ell+1,\\
\Phi_{\ell,\ell+1} &\text{ if }i=s.
\end{cases}
\end{equation*}
We prove in the next paragraph that
\begin{equation}\label{phi'}
\bigcup_{j=i}^{b}(\Phi_{\ell+1,i}\cap\Phi_{\ell+1,j}^{-1})\neq\emptyset\ \text{ for all }\,i\in\{2,\ldots,\ell+1\}.
\end{equation}

For each $i\in\{2,\ldots,\ell+1\}$, noticing that $(\Phi_{\ell+1,i},\ldots,\Phi_{\ell+1,b})$ is an reordering of $(\Phi_{\ell,i},\ldots,\Phi_{\ell,b})$, we have $\cup_{j=i}^{b}\Phi_{\ell+1,j}^{-1}=\cup_{j=i}^{b}\Phi_{\ell,j}^{-1}$, which yields that
\begin{equation}\label{gamma'}
\bigcup_{j=i}^{b}(\Phi_{\ell+1,i}\cap\Phi_{\ell+1,j}^{-1})=\Phi_{\ell+1,i}\cap\bigcup_{j=i}^{b}\Phi_{\ell+1,j}^{-1}=\Phi_{\ell+1,i}\cap\bigcup_{j=i}^{b}\Phi_{\ell,j}^{-1}=\bigcup_{j=i}^{b}(\Phi_{\ell+1,i}\cap\Phi_{\ell,j}^{-1}).
\end{equation}
For $i\in\{2,\ldots,\ell\}$, since $\Phi_{\ell+1,i}=\Phi_{\ell,i}$, it follows that
\[
\bigcup_{j=i}^{b}(\Phi_{\ell+1,i}\cap\Phi_{\ell+1,j}^{-1})=\bigcup_{j=i}^{b}(\Phi_{\ell,i}\cap\Phi_{\ell,j}^{-1})\neq\emptyset.
\]
As for $i=\ell+1$, according to~\eqref{gamma'} and the relation $\Phi_{\ell+1,\ell+1}=\Phi_{\ell,s}$, we have
\[
\bigcup_{j=\ell+1}^{b}(\Phi_{\ell+1,\ell+1}\cap\Phi_{\ell+1,j}^{-1})=\bigcup_{j=\ell+1}^{b}(\Phi_{\ell,s}\cap\Phi_{\ell,j}^{-1}).
\]
This together with $\cup_{j=\ell+1}^{b}(\Phi_{\ell,s}\cap\Phi_{\ell,j}^{-1})\neq\emptyset$ in~\eqref{smallests} implies that $\cup_{j=\ell+1}^{b}(\Phi_{\ell+1,\ell+1}\cap\Phi_{\ell+1,j}^{-1})\neq\emptyset$.

Now we have shown that the ordering $(\Phi_{\ell+1,1},\Phi_{\ell+1,2},\ldots,\Phi_{\ell+1,b})$ constructed from the ordering $(\Phi_{\ell,1},\Phi_{\ell,2},\ldots,\Phi_{\ell,b})$ satisfies~\eqref{phi'}. By induction, the ordering
\[
(\Lambda_{1},\Lambda_{2},\ldots,\Lambda_{b}):=(\Phi_{\lfloor b/2\rfloor+1,1},\Phi_{\lfloor b/2\rfloor+1,2}\ldots,\Phi_{\lfloor b/2\rfloor+1,b})
\]
satisfies the requirement of the lemma.
\end{proof}

\begin{lemma}\label{lem:nsM}
Let $R$ be a group, let $QxQ$ be an inverse-closed double coset of a subgroup $Q$ in $R$ with $|QxQ|=b|Q|$, and let
\[
\mathcal{M}=\{S\subseteq QxQ: S=S^{-1},\, S\text{ intersects $QxQ$ evenly}\}.
\]
We have $|\mathcal{M}|\leq2^{\mathbf{c}(QxQ)-\frac{1}{2}b+\frac{1}{2}}$.
\end{lemma}

\begin{proof}
If $b=1$, then the set $\mathcal{M}$ is exactly the set of inverse-closed subsets of $QxQ$. In this case, we have $|\mathcal{M}|=2^{\mathbf{c}(QxQ)}$, which meets the upper bound. For the rest of the proof we assume $b\geq2$.

Let $(\Lambda_{1},\Lambda_{2},\ldots,\Lambda_{b})$ be an ordering of the $b$ right $Q$-cosets in $QxQ$. Write
\[
\Lambda_{i,j}=\Lambda_{i}\cap\Lambda_{j}^{-1}\ \text{ for }\,i,j\in\{1,2,\ldots,b\}.
\]
The reader may find Figure~\ref{tab0} helpful.

\vspace{4mm}
\begin{figure}[h]
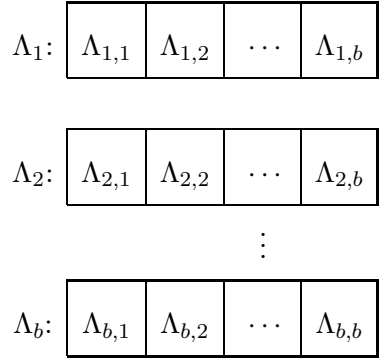

 \renewcommand\arraystretch{2}
 \centering
 \begin{tabular}{ccccc}
 \cline{2-5}
 \multicolumn{1}{c|}{$\Lambda_{1}$:~~~~} & \multicolumn{1}{c|}{$\Lambda_{1,1}$} & \multicolumn{1}{c|}{$\Lambda_{1,2}$} & \multicolumn{1}{c|}{~$\cdots$~} & \multicolumn{1}{c|}{$\Lambda_{1,b}$}
 \\ \cline{2-5}
 \\[-2ex] \cline{2-5}
 \multicolumn{1}{c|}{$\Lambda_{2}$:~~~~} & \multicolumn{1}{c|}{$\Lambda_{2,1}$} & \multicolumn{1}{c|}{$\Lambda_{2,2}$} & \multicolumn{1}{c|}{~~~$\cdots$~~~} & \multicolumn{1}{c|}{$\Lambda_{2,b}$}
 \\ \cline{2-5}
 \multicolumn{1}{c}{} & \multicolumn{1}{c}{} & \multicolumn{1}{c}{}  & \multicolumn{1}{c}{$\vdots$} & \multicolumn{1}{c}{}
 \\ \cline{2-5}
 \multicolumn{1}{c|}{$\Lambda_{b}$:~~~~} & \multicolumn{1}{c|}{$\Lambda_{b,1}$} & \multicolumn{1}{c|}{$\Lambda_{b,2}$} & \multicolumn{1}{c|}{~$\cdots$~} & \multicolumn{1}{c|}{$\Lambda_{b,b}$}
 \\ \cline{2-5}
 \end{tabular}
\vspace{3mm}
 \caption{The $b$ right $Q$-cosets in $QxQ$}
 \label{tab0}
\end{figure}

By Lemma~\ref{Claimforarrangement}, we can assume that
\begin{equation}\label{assumption}
\bigcup_{j=i}^{b}\Lambda_{i,j}\neq\emptyset\ \text{ for all }\,i\in\{2,\ldots,\lfloor b/2\rfloor+1\}.
\end{equation}
It is clear that $\Lambda_{i,j}^{-1}=\Lambda_{j,i}$ for $i,j\in\{1,2,\ldots,b\}$. In particular, $\Lambda_{i,i}$ is inverse-closed.
Since $|\Lambda_{i,j}|=|\Lambda_{i,j}^{-1}|=|\Lambda_{j,i}|$ and every element of order at most $2$ in $QxQ$ lies in some $\Lambda_{i,i}$, one has
\begin{equation}\label{Claimfortotal}
\sum_{i=1}^{b}\mathbf{c}(\Lambda_{i,i})+\sum_{i=1}^{b-1}\sum_{j=i+1}^{b}|\Lambda_{i,j}|=\mathbf{c}(QxQ).
\end{equation}

To determine each $S\in\mathcal{M}$ (and count the number of possibilities for $S$), we determine $S\cap\Lambda_1,S\cap\Lambda_2,\ldots,S\cap\Lambda_b$ in turn. First we determine $S\cap\Lambda_1$. Note that this is further determined by $S\cap\Lambda_{1,1}$ and $S\cap\cup_{j=2}^{b}\Lambda_{1,j}$. We first choose an inverse-closed subset of $\Lambda_{1,1}$ and then choose an arbitrary subset of $\cup_{j=2}^{b}\Lambda_{1,j}$. Thus the number of choices for $S\cap\Lambda_{1}$ is
\begin{equation}\label{clamda1}
2^{\mathbf{c}(\Lambda_{1,1})}\cdot2^{|\cup_{j=2}^{b}\Lambda_{1,j}|}=2^{\mathbf{c}(\Lambda_{1,1})+\sum_{j=2}^{b}|\Lambda_{1,j}|}.
\end{equation}
In the following, we estimate the number of choices for $S\cap\Lambda_{i}$ with $i\geq 2$.
Note that once $S\cap\Lambda_1$ has been determined, the size of $S\cap\Lambda_i$ for every $i\in\{2,\ldots,b\}$ has to be $|S\cap\Lambda_1|$ since we require that $S$ intersects $QxQ$ evenly.

Now assume that $\{S\cap\Lambda_1,\ldots,S\cap\Lambda_{i-1}\}$ has been determined for some $i\in\{2,\ldots,b\}$.
For an inverse-closed subset $S$ of $R$, we have $S\cap\Lambda_{i,j}=S^{-1}\cap\Lambda_{j,i}^{-1}=(S\cap\Lambda_{j,i})^{-1}$ for all $j\in\{1,2,\ldots,b\}$. In particular, for each $j\in\{1,\ldots,i-1\}$, the set $S\cap\Lambda_{i,j}$ is determined as $S\cap\Lambda_{j,i}$ is already determined, see Figure~\ref{tab1}.
Therefore, to determine $S\cap\Lambda_i$, we only need to determine $\cup_{j=i}^{b}(S\cap\Lambda_{i,j})$ such that
\[
\Big|\bigcup_{j=i}^{b}(S\cap\Lambda_{i,j})\Big|=|S\cap\Lambda_{i}|-\Big|\bigcup_{j=1}^{i-1}(S\cap\Lambda_{i,j})\Big|=|S\cap\Lambda_{1}|-\Big|\bigcup_{j=1}^{i-1}(S\cap\Lambda_{j,i})\Big|.
\]
This means that we must choose a subset $\cup_{j=i}^{b}(S\cap\Lambda_{i,j})$ of $\cup_{j=i}^{b}\Lambda_{i,j}$ of determined size.
We estimate the number of such choices in the following two paragraphs, according to $2\leq i\leq\lfloor b/2\rfloor+1$ or $i>\lfloor b/2\rfloor+1$, respectively.
\vspace{4mm}
\begin{figure}[h]
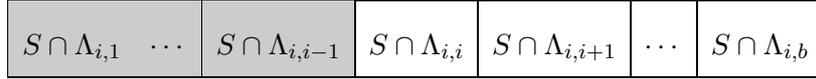

 \renewcommand\arraystretch{2}
 \centering
 \begin{tabular}{ccccccc}
 \hline
 \multicolumn{1}{|>{\columncolor[gray]{0.8}}c|}{$S\cap\Lambda_{i,1}$} & \multicolumn{1}{>{\columncolor[gray]{0.8}}c|}{$\cdots$} & \multicolumn{1}{>{\columncolor[gray]{0.8}}c|}{$S\cap\Lambda_{i,i-1}$} & \multicolumn{1}{c|}{$S\cap\Lambda_{i,i}$} & \multicolumn{1}{c|}{$S\cap\Lambda_{i,i+1}$} & \multicolumn{1}{c|}{$\cdots$} & \multicolumn{1}{c|}{$S\cap\Lambda_{i,b}$}
 \\
 \hline
 \end{tabular}
 \vspace{3mm}
 \caption{$S\cap\Lambda_i$}
 \label{tab1}
\end{figure}

First assume that $2\leq i\leq\lfloor b/2\rfloor+1$. Then $\cup_{j=i}^{b}\Lambda_{i,j}\neq\emptyset$ by~\eqref{assumption}, and we are going to determine $\cup_{j=i}^{b}(S\cap\Lambda_{i,j})$ in the two cases $\cup_{j=i+1}^{b}\Lambda_{i,j}=\emptyset$ and $\cup_{j=i+1}^{b}\Lambda_{i,j}\neq\emptyset$, respectively.
If $\cup_{j=i+1}^{b}\Lambda_{i,j}=\emptyset$, then $\Lambda_{i,i}\neq\emptyset$ and we only need to choose an inverse-closed subset of $\Lambda_{i,i}$ of determined size. In this case, by Lemma~\ref{lem:nics-givensize}, the number of choices for $S\cap\Lambda_i$ is at most
\[
2^{\mathbf{c}(\Lambda_{i,i})-1}=2^{\mathbf{c}(\Lambda_{i,i})+\sum_{j=i+1}^{b}|\Lambda_{i,j}|-1}.
\]
If $\cup_{j=i+1}^{b}\Lambda_{i,j}\neq\emptyset$, then we need to first choose an inverse-closed subset of $\Lambda_{i,i}$ (no matter whether $\Lambda_{i,i}$ is empty or not) and then choose a subset of the nonempty set $\cup_{j=i+1}^{b}\Lambda_{i,j}$ of the determined size $|S\cap\Lambda_{1}|-|\cup_{j=1}^{i-1}(S\cap\Lambda_{j,i})|-|S\cap\Lambda_{i,i}|$. In this case, by~\eqref{eq:BoundBinom}, the number of choices for $S\cap\Lambda_i$ is at most
\[
2^{\mathbf{c}(\Lambda_{i,i})}\cdot2^{|\cup_{j=i+1}^{b}\Lambda_{i,j}|-1}=2^{\mathbf{c}(\Lambda_{i,i})+\sum_{j=i+1}^{b}|\Lambda_{i,j}|-1}.
\]
In either case, we have at most
\begin{equation}\label{clamda2tohalf}
2^{\mathbf{c}(\Lambda_{i,i})+\sum_{j=i+1}^{b}|\Lambda_{i,j}|-1}
\end{equation}
choices for $S\cap\Lambda_i$.

Next assume that $i>\lfloor b/2\rfloor+1$. To determine $\cup_{j=i}^{b}(S\cap\Lambda_{i,j})$, we need to first choose an inverse-closed subset of $\Lambda_{i,i}$ and then choose a subset of $\cup_{j=i+1}^{b}\Lambda_{i,j}$. Hence the number of choices for $S\cap\Lambda_{i}$ is at most
\begin{equation}\label{clamdahalftob}
2^{\mathbf{c}(\Lambda_{i,i})+\sum_{j=i+1}^{b}|\Lambda_{i,j}|}.
\end{equation}

Now combining the bounds in~\eqref{clamda1}, \eqref{clamda2tohalf} and~\eqref{clamdahalftob} for $i\in\{1,2,\ldots,b\}$, we obtain
\begin{align*}
|\mathcal{M}|&\leq2^{\mathbf{c}(\Lambda_{1,1})+\sum_{j=2}^{b}|\Lambda_{1,j}|}\cdot\prod_{i=2}^{\lfloor b/2\rfloor+1}2^{\mathbf{c}(\Lambda_{i,i})+\sum_{j=i+1}^{b}|\Lambda_{i,j}|-1}\cdot\prod_{i=\lfloor b/2\rfloor+2}^b2^{\mathbf{c}(\Lambda_{i,i})+\sum_{j=i+1}^{b}|\Lambda_{i,j}|}\\
&=2^{\sum_{i=1}^{b}\mathbf{c}(\Lambda_{i,i})+\sum_{i=1}^{b}\sum_{j=i+1}^{b}|\Lambda_{i,j}|-\lfloor b/2\rfloor}\\
&\leq2^{\sum_{i=1}^{b}\mathbf{c}(\Lambda_{i,i})+\sum_{i=1}^{b}\sum_{j=i+1}^{b}|\Lambda_{i,j}|-\frac{1}{2}b+\frac{1}{2}}.
\end{align*}
This together with~\eqref{Claimfortotal} yields that $|\mathcal{M}|\leq2^{\mathbf{c}(QxQ)-\frac{1}{2}b+\frac{1}{2}}$ as desired.
\end{proof}

\begin{lemma}\label{lem:nsN}
Let $R$ be a group, let $QxQ$ be a double coset of a subgroup $Q$ in $R$ with $|QxQ|=b|Q|$, and let
\[
\mathcal{N}=\{S\subseteq QxQ: S\text{ intersects $QxQ$ evenly}\}.
\]
We have $|\mathcal{N}|\leq2^{|QxQ|-b+1}$.
\end{lemma}

\begin{proof}
Let the $b$ right $Q$-cosets in $QxQ$ be $\Lambda_{1},\Lambda_{2},\ldots,\Lambda_{b}$. Note that a subset $S\in\mathcal{N}$ is determined by $\{S\cap\Lambda_i: 1\leq i\leq b \}$. To determine $S\cap\Lambda_1$, we may choose an arbitrary subset of $\Lambda_{1}$, which means that we have $2^{|Q|}$ choices for $S\cap\Lambda_1$. Note that once $S\cap\Lambda_{1}$ is determined, the size of $S\cap\Lambda_{i}$ for every $i\in\{2,\ldots,b\}$ is also determined. Thus, by~\eqref{eq:BoundBinom}, there are at most $2^{|Q|-1}$ choices for each of $S\cap\Lambda_2,\dots,S\cap\Lambda_b$. Consequently,
\[
|\mathcal{N}|\leq 2^{|Q|}\cdot(2^{|Q|-1})^{b-1}=2^{b|Q|-b+1}=2^{|QxQ|-b+1},
\]
as desired.
\end{proof}

\begin{lemma}\label{lem:nsL}
Let $R$ be a group, let $Q$ be a subgroup of $R$ with $|Q\backslash R/Q|=\ell$, and let
\[
\mathcal{L}=\{S\subseteq R: S=S^{-1},\, S \text{ intersects $\Delta$ evenly for each } \Delta\in Q\backslash R/Q\}.
\]
We have $|\mathcal{L}|\leq 2^{\mathbf{c}(R)-\frac{1}{2}|R\,{:}\,Q|+\frac{1}{2}\ell}$.
\end{lemma}

\begin{proof}
We divide $Q\backslash R/Q$ into the following two subsets:
\begin{align*}
  &\mathcal{D}_1=\{\Delta\in Q\backslash R/Q: \Delta=\Delta^{-1}\},\\
  &\mathcal{D}_2=\{\Delta\in Q\backslash R/Q: \Delta\neq\Delta^{-1}\}.
\end{align*}
Write $\mathcal{D}_1=\{A_1,\ldots,A_s\}$. For each $i\in \{1,\ldots,s\}$, let $a_i=|A_i|/|Q|$, that is, the number of right $Q$-cosets in $A_i$.
Write $\mathcal{D}_2=\{B_1,\ldots,B_{2t}\}$ with $B_i^{-1}=B_{i+t}$ for $i\in \{1,\ldots,t\}$. For each $i\in \{1,\ldots,t\}$, let $b_i=|B_i|/|Q|$. It is clear that we have
\begin{equation}\label{t+2m}
s+2t=|Q\backslash R/Q|=\ell,
\end{equation}
\begin{equation}\label{a+2b}
\sum_{i=1}^{s}a_i+2\sum_{i=1}^{t}b_i=|R\,{:}\,Q|,
\end{equation}
\begin{equation}\label{alpha+2beta+i}
\left(\sum_{i=1}^{s}|A_i|+2\sum_{i=1}^{t}|B_i|\right)+\sum_{i=1}^{s}|\I(A_i)|=|R|+|\I(R)|=2\mathbf{c}(R).
\end{equation}

For an inverse-closed subset $S$ of $R$, we have $(S\cap\Delta)^{-1}=S^{-1}\cap\Delta^{-1}=S\cap\Delta^{-1}$ for every $\Delta\in Q\backslash R/Q$, which means that $S\cap\Delta^{-1}$ is determined by $S\cap\Delta$ when $\Delta\in\mathcal{D}_2$. Thus a subset $S\in\mathcal{L}$ is determined by
\[
S\cap A_1,\dots,S\cap A_s,S\cap B_1,\dots,S\cap B_t.
\]
According to Lemma~\ref{lem:nsM}, for each $i\in \{1,\ldots,s\}$, there are at most $2^{\mathbf{c}(A_i)-\frac{1}{2}a_i+\frac{1}{2}}$ choices for $S\cap A_i$. Hence the number of choices for $S\cap A_1,\dots,S\cap A_s$ is at most
\[
\prod_{i=1}^{s}2^{\mathbf{c}(A_i)-\frac{1}{2}a_i+\frac{1}{2}}=2^{\sum\limits_{i=1}^{s}\mathbf{c}(A_i)-\frac{1}{2}\sum\limits_{i=1}^{s}a_i+\frac{1}{2}s}.
\]
By Lemma~\ref{lem:nsN}, for each $i\in\{1,\ldots,t\}$, we have at most $2^{|B_i|-b_i+1}$ choices for $S\cap B_i$. This implies that the number of choices for $S\cap B_1,\dots,S\cap B_t$ is at most
\[
\prod_{i=1}^{t}2^{|B_i|-b_i+1}=2^{\sum\limits_{i=1}^{t}|B_i|-\sum\limits_{i=1}^{t}b_i+t}.
\]
It follows that
\begin{align*}
  |\mathcal{L}| & \leq 2^{\sum\limits_{i=1}^{s}\mathbf{c}(A_i)-\frac{1}{2}\sum\limits_{i=1}^{s}a_i+\frac{1}{2}s}\cdot2^{\sum\limits_{i=1}^{t}|B_i|-\sum\limits_{i=1}^{t}b_i+t}\\
  &=2^{\left(\frac{1}{2}\sum\limits_{i=1}^{s}|A_i|+\frac{1}{2}\sum\limits_{i=1}^{s}|\I(A_i)|\right)+\sum\limits_{i=1}^{t}|B_i|-(\frac{1}{2}\sum\limits_{i=1}^{s}a_i+\sum\limits_{i=1}^{t}b_i)+(\frac{1}{2}s+t)}\\
  &=2^{\frac{1}{2}\left(\sum\limits_{i=1}^{s}|A_i|+2\sum\limits_{i=1}^{t}|B_i|\right)+\frac{1}{2}\sum\limits_{i=1}^{s}|\I(A_i)|-(\frac{1}{2}\sum\limits_{i=1}^{s}a_i+\sum\limits_{i=1}^{t}b_i)+(\frac{1}{2}s+t)}.
\end{align*}
This together with~\eqref{t+2m}, \eqref{a+2b} and~\eqref{alpha+2beta+i} yields
\[
|\mathcal{L}|\leq 2^{\mathbf{c}(R)-\frac{1}{2}|R\,{:}\,Q|+\frac{1}{2}\ell},
\]
which completes the proof.
\end{proof}

We are now ready to prove the main result of this subsection.

\begin{proof}[Proof of Proposition~\ref{prop:ns-nonnorm}]
Let $\Omega=R/Q=\{Qx: x\in R\}$ be the set of right cosets of $Q$ in $R$. Then $R$ acts transitively by right multiplication on $\Omega$ with stabilizer $Q$. Note that the stabilizer $Q$ does not act trivially on $\Omega$, for otherwise we would have $(x^{-1}Qx)\cap Q=Q$ for every $x\in R$, contradicting the assumption that $Q$ is non-normal in $R$. This implies that $R$ is not regular on $\Omega$. Let $\ell$ be the number of orbits of $Q$ on $\Omega$, that is, the number of double cosets of $Q$ in $R$. According to~\cite[Lemma 3.1]{Morris-S2021}, we have $\ell\leq\frac{3}{4}|\Omega|=\frac{3}{4}|R\,{:}\,Q|$.
Hence by Lemma~\ref{lem:nsL}, we have
\[
|\mathcal{L}|\leq2^{\mathbf{c}(R)-\frac{1}{2}|R\,{:}\,Q|+\frac{1}{2}\ell}\leq2^{\mathbf{c}(R)-\frac{1}{2}|R\,{:}\,Q|+\frac{3}{8}|R\,{:}\,Q|}=2^{\mathbf{c}(R)-\frac{1}{8}|R\,{:}\,Q|},
\]
as desired.
\end{proof}

\subsection{Cosets from an orbit of a stabilizer}

Let $G$ be a finite transitive permutation group on $R$ that properly contains the right regular representation of $R$, let $G_1$ be the stabilizer of $1\in R$ in $G$, let
\[
K=\bigcap_{g\in G}R^g
\]
be the core of $R$ in $G$,
and let $\overline{\phantom{u}}\colon R\to R/K$ be the quotient modulo $K$. Then $G/K$ acts naturally on $R/K$ and induces a transitive permutation group containing the right regular representation of $R/K$, and
\[
H:=G_1K/K
\]
is the stabilizer of $\overline{1}\in R/K$ in $G/K$.
For a subset $\Delta$ of $R/K$, the full preimage of $\Delta$ in $R$ under $\overline{\phantom{u}}$, denoted by $\widetilde{\Delta}$, is the union of the right $K$-cosets in $\Delta$.

\begin{proposition}\label{prop:ns-G1}
With the above notation, suppose that $R/K$ is neither abelian of exponent greater than $2$ nor generalized dicyclic, and let
\[
\mathcal{L}=\{S\subseteq R:S=S^{-1},\,S \text{ intersects $\widetilde{\Delta}$ evenly for each } H\text{-orbit }\Delta\text{ on } R/K\}.
\]
We have $|\mathcal{L}|\leq 2^{\mathbf{c}(R)-\frac{1}{96}|R/K|}$.
\end{proposition}

In the rest of this subsection we embark on the proof of Proposition~\ref{prop:ns-G1} and fix the notation $R$, $K$, $H$, $\overline{\phantom{u}}$ and $\widetilde{\phantom{u}}$. Moreover, let $\iota$ be the permutation on $R/K$ sending every element to its inverse.

\begin{lemma}\label{intwithicHorb}
Let $\Delta\subseteq R/K$ be an $H$-orbit on $R/K$ such that $\Delta^\iota=\Delta$, and let
\[
\mathcal{M}=\{S\subseteq\widetilde{\Delta}: S=S^{-1},\, S\text{ intersects }\widetilde{\Delta}\text{ evenly}\}.
\]
We have $|\mathcal{M}|\leq2^{\mathbf{c}(\widetilde{\Delta})-\mathbf{c}(\Delta)+1}$.
\end{lemma}

\begin{proof}
Write $\mathcal{I}(\Delta)=\{\overline{x_1},\ldots,\overline{x_n}\}$ and $\Delta\setminus \mathcal{I}(\Delta)=\{\overline{y_1},\ldots,\overline{y_{2m}}\}$ with $(\overline{y_j})^{-1}=\overline{y_{m+j}}$ for $j\in\{1,\dots,m\}$. It is clear that
\[
n+m=\frac{n}{2}+\frac{n+2m}{2}=\frac{|\mathcal{I}(\Delta)|+|\Delta|}{2}=\mathbf{c}(\Delta),
\]
and that
\[
\sum_{i=1}^{n}\mathbf{c}(Kx_i)+\sum_{j=1}^{m}|Ky_j|=\mathbf{c}(\widetilde{\Delta}).
\]
For $i\in\{1,\dots,n\}$, since $\overline{x_i}$ has order at most $2$, we have $Kx_i^{-1}=Kx_i$.
For an inverse-closed subset $S$ of $R$, we have $(S\cap Kz)^{-1}=S^{-1}\cap Kz^{-1}=S\cap Kz^{-1}$ for every $\overline{z}\in\Delta$,
which means that $S\cap Kz^{-1}$ is determined by $S\cap Kz$ when $\overline{z}\in\Delta\setminus \mathcal{I}(\Delta)$.
Thus an inverse-closed subset $S$ of $\widetilde{\Delta}$ is determined by
\[
S\cap Kx_1,\dots,S\cap Kx_n,S\cap Ky_1,\dots,S\cap Ky_m.
\]

First assume $n\geq 1$.
Consider $S$ in $\mathcal{M}$. To determine $S\cap Kx_1$, we may choose an inverse-closed subset of $Kx_1$, which means that there are $2^{\mathbf{c}(Kx_1)}$ choices for $S\cap Kx_1$. Since $S$ intersects $\widetilde{\Delta}$ evenly, it follows that, once $S\cap Kx_1$ is determined, the size of $S\cap Kz$ for every $\overline{z}\in\Delta\setminus \{\overline{x_1}\}$ is also determined. Thus, Lemma~\ref{lem:nics-givensize} implies that there are at most $2^{\mathbf{c}(Kx_i)-1}$ choices for each $S\cap Kx_i$ with $i\in\{2,\dots,n\}$, and~\eqref{eq:BoundBinom} implies that there are at most $2^{|Ky_j|-1}$ choices for each $S\cap Ky_j$ with $j\in\{1,\ldots,m\}$. Consequently,
\begin{align*}
|\mathcal{M}|&\leq 2^{\mathbf{c}(Kx_1)}\cdot\prod_{i=2}^{n}2^{\mathbf{c}(Kx_i)-1}\cdot\prod_{j=1}^{m}2^{|Ky_j|-1}\\
&=2^{\sum_{i=1}^{n}\mathbf{c}(Kx_i)-(n-1)+\sum_{j=1}^{m}|Ky_j|-m}\\
&=2^{\sum_{i=1}^{n}\mathbf{c}(Kx_i)+\sum_{j=1}^{m}|Ky_j|-\mathbf{c}(\Delta)+1}\\
&=2^{\mathbf{c}(\widetilde{\Delta})-\mathbf{c}(\Delta)+1}.
\end{align*}

Now assume $n=0$ and consider $S$ in $\mathcal{M}$. To determine $S\cap Ky_1$, one may choose a subset of $Ky_1$, which means that there are $2^{|Ky_1|}$ choices for $S\cap Ky_1$. Note that once $S\cap Ky_1$ is determined, the size of $S\cap Kz$ for every $\overline{z}\in\Delta\setminus \{\overline{y_1}\}$ is also determined. Thus, by~\eqref{eq:BoundBinom}, there are at most $2^{|Ky_j|-1}$ choices for each $S\cap Ky_j$ with $j\in\{2,\ldots,m\}$. This leads to
\[
|\mathcal{M}|\leq2^{|Ky_1|}\cdot\prod_{j=2}^{m}2^{|Ky_j|-1}=2^{\sum_{j=1}^{m}|Ky_j|-(m-1)}=2^{\mathbf{c}(\widetilde{\Delta})-\mathbf{c}(\Delta)+1}.\qedhere
\]
\end{proof}

\begin{lemma}\label{intwithHorb}
Let $\Delta\subseteq R/K$ be an $H$-orbit on $R/K$, and let
\[
\mathcal{N}=\{S\subseteq\widetilde{\Delta}: S\text{ intersects }\widetilde{\Delta}\text{ evenly}\}.
\]
We have $|\mathcal{N}|\leq2^{|\widetilde{\Delta}|-|\Delta|+1}$.
\end{lemma}

\begin{proof}
Write $\Delta=\{\overline{x_1},\ldots,\overline{x_n}\}$, where $n=|\Delta|$. Consider $S$ in $\mathcal{M}$.
As $S$ intersects $\widetilde{\Delta}$ evenly, we derive that, once $S\cap Kx_1$ is determined, the size of $S\cap Kx$ for every $\overline{x}\in\Delta\setminus \{\overline{x_1}\}$ is also determined.
Hence we deduce from~\eqref{eq:BoundBinom} that
\[
|\mathcal{N}|\leq2^{|Kx_1|}\cdot\prod_{i=2}^n2^{|Kx_i|-1}=2^{\sum_{i=1}^n|Kx_i|-(n-1)}=2^{|\widetilde{\Delta}|-|\Delta|+1},
\]
as desired.
\end{proof}

\begin{lemma}\label{intwithallHorb}
Let $\kappa$ be the number of $\langle H,\iota\rangle$-orbits on $R/K$, and let
\[
\mathcal{L}=\{S\subseteq R: S=S^{-1},\, S \text{ intersects $\widetilde{\Delta}$ evenly for each } H\text{-orbit }\Delta\text{ on } R/K\}.
\]
We have $|\mathcal{L}|\leq 2^{\mathbf{c}(R)-\mathbf{c}(R/K)+\kappa}$.
\end{lemma}

\begin{proof}
Let $\mathcal{D}$ be the set of $H$-orbits on $R/K$.
It is well known and not difficult to prove (see, for example,~\cite[Theorem~24.1]{Wielandt1964}) that the set of $H$-orbits on $R/K$ is $\iota$-invariant.
We divide $\mathcal{D}$ into the following two subsets:
\begin{align*}
\mathcal{D}_1&=\{\Delta\in\mathcal{D}: \Delta^\iota=\Delta\},\\
\mathcal{D}_2&=\{\Delta\in\mathcal{D}: \Delta^\iota\neq\Delta\}.
\end{align*}
Note that the elements in $\mathcal{D}_2$ are paired by $\iota$.
Write $\mathcal{D}_1=\{A_1,\ldots,A_s\}$ and $\mathcal{D}_2=\{B_1,\ldots,B_{2t}\}$ with $(B_j)^\iota=B_{t+j}$ for $j\in\{1,\dots,t\}$. Clearly,
\[
s+t=\kappa,
\]
\[
\sum_{i=1}^{s}\mathbf{c}(A_i)+\sum_{j=1}^{t}|B_j|=\sum_{i=1}^{s}\mathbf{c}(A_i)+\sum_{j=1}^{2t}\mathbf{c}(B_j)=\mathbf{c}(R/K),
\]
\[
\sum_{i=1}^{s}\mathbf{c}(\widetilde{A_i})+\sum_{j=1}^{t}|\widetilde{B_j}|=\sum_{i=1}^{s}\mathbf{c}(\widetilde{A_i})+\sum_{j=1}^{2t}\mathbf{c}(\widetilde{B_j})=\mathbf{c}(R).
\]

For an inverse-closed subset $S$ of $R$, we have $(S\cap\widetilde{\Delta})^{-1}=S\cap\widetilde{\Delta}^{-1}$ for every $\Delta\in\mathcal{D}$,
which means that $S\cap\widetilde{\Delta}^{-1}$ is determined by $S\cap\widetilde{\Delta}$ when $\Delta\in\mathcal{D}_2$. Thus a subset $S\in\mathcal{L}$ is determined by
\[
S\cap\widetilde{A_1},\dots,S\cap\widetilde{A_s},S\cap\widetilde{B_1},\dots,S\cap\widetilde{B_t}.
\]
Since $S$ intersects $\widetilde{\Delta}$ evenly for all $\Delta\in\mathcal{D}$, we conclude by Lemmas~\ref{intwithicHorb} and~\ref{intwithHorb} that
\begin{align*}
  |\mathcal{L}| & \leq \prod_{i=1}^{s}2^{\mathbf{c}(\widetilde{A_i})-\mathbf{c}(A_i)+1}\cdot\prod_{j=1}^{t}2^{|\widetilde{B_j}|-|B_j|+1}\\
  & =2^{(\sum_{i=1}^{s}\mathbf{c}(\widetilde{A_i})+\sum_{j=1}^{t}|\widetilde{B_j}|)-(\sum_{i=1}^{s}\mathbf{c}(A_i)+\sum_{j=1}^{t}|B_j|)+(s+t)}\\
  & =2^{\mathbf{c}(R)-\mathbf{c}(R/K)+\kappa},
\end{align*}
which completes the proof.
\end{proof}

\begin{proof}[Proof of Proposition~\ref{prop:ns-G1}]
It follows from Lemma~\ref{lem:GTS} that the number $\kappa$ of $\langle H, \iota\rangle$-orbits on $R/K$ is at most $\mathbf{c}(R/K)-\frac{1}{96}|R/K|$. Applying Lemma~\ref{intwithallHorb}, we obtain
\[
|\mathcal{L}|\leq 2^{\mathbf{c}(R)-\mathbf{c}(R/K)+\kappa}\leq2^{\mathbf{c}(R)-\frac{1}{96}|R/K|},
\]
completing the proof.
\end{proof}

\section{Reduction}\label{Sec4}

Let $R$ be a group of order $r$ that is neither abelian of exponent greater than $2$ nor generalized dicyclic. We also use $R$ to represent the right regular representation of the group $R$. Let
\begin{align*}
\mathcal{Z}&=\{S\subseteq R: S^{-1}=S,\,\Aut(\Cay(R,S))>R\}\\
&=\{S\subseteq R: S^{-1}=S,\,\text{there exists }G\leq\Aut(\Cay(R,S))\text{ with }R\text{ maximal in }G\}.
\end{align*}
Clearly, $\mathcal{Z}$ is the set of inverse-closed subsets $S$ of $R$ such that $\Cay(R,S)$ is not a GRR of $R$, and the key to prove Theorem~\ref{Thm:main} is to establish a suitable upper bound for $|\mathcal{Z}|$.

Let $S$ be a subset of $R$ in $\mathcal{Z}$. There exists some group $G\leq\Aut(\Cay(R,S))$ such that $R$ is a maximal subgroup of $G$. Denote the stabilizer of the vertex $1$ in $G$ by $G_1$, and denote the core of $R$ in $G$ by $\Core_G(R)$, that is,
\[
\Core_G(R)=\bigcap_{g\in G}R^g.
\]
It is clear that $G/\Core_G(R)$ acts primitively and faithfully by right multiplication on the set $G/R$ of right cosets of $R$ in $G$ and that $G_1\Core_G(R)/\Core_G(R)\cong G_1$ acts regularly on $G/R$.
To sum up, for each $S\in\mathcal{Z}$, there exists a group $G=G(S)$ such that $G/\Core_G(R)$ is a primitive permutation group on $G/R$ with point stabilizer $R/\Core_G(R)$ and a regular subgroup $G_1\Core_G(R)/\Core_G(R)\cong G_1$.

Now we have chosen a group $G(S)$ for each $S\in\mathcal{Z}$.
For $\varepsilon\in(0,0.5)$ let
\begin{align*}
\mathcal{Y}_\varepsilon&=\{S\in\mathcal{Z}:|G(S)_1|>2^{r^{0.5-\varepsilon}}\},\\
\mathcal{X}_\varepsilon&=\{S\in\mathcal{Y}_\varepsilon:|\Core_{G(S)}(R)|\leq96\log_2r\},\\
\mathcal{W}_\varepsilon&=\{S\in\mathcal{X}_\varepsilon:\text{every orbit of }\Core_{G(S)}(R)\text{ is stabilized by }G(S)_1\},\\
\mathcal{T}_\varepsilon&=\{S\in\mathcal{X}_\varepsilon:\text{some orbit of }\Core_{G(S)}(R)\text{ is not stabilized by }G(S)_1\}.
\end{align*}
We omit the subscript $\varepsilon$ when we take $\varepsilon=0.001$.



The following result is obtained by combining Lemma~\ref{lem:GTS} with~\cite[Lemma 3.2]{Morris-S2021}. It is worth remarking that~\cite[Lemma 3.2]{Morris-S2021} depends upon the following asymptotic enumeration due to Lubotzky~\cite{Lubotzky2001} (see the third line on page~416 of~\cite{Morris-S2021}): The number of isomorphism classes of groups of order $n$ that are $d$-generated is at most $n^{2(d+1)\log_2n}$.

\begin{lemma}\label{lem:Gsmall}
Let $R$ be a group of order $r$ that is neither abelian of exponent greater than $2$ nor generalized dicyclic, and let $\varepsilon\in (0,0.5)$. If $r$ is large enough, then
\[
|\mathcal{Z}\setminus\mathcal{Y}_\varepsilon|\leq2^{\mathbf{c}(R)-\frac{r}{96}+r^{1-\varepsilon}}.
\]
\end{lemma}

\begin{proof}
Let $\mathcal{G}_\varepsilon=\{G(S):S\in\mathcal{Z}\setminus\mathcal{Y}_\varepsilon\}$. Assume that $r$ is sufficiently large. Since
\[
\mathcal{G}_{\varepsilon}\subseteq\{G\leq \Sym(R): R \textrm{ is maximal in }G, \,|G_1|\leq 2^{r^{0.5-\varepsilon}}\},
\]
we derive from~\cite[Lemma 3.2]{Morris-S2021} that $|\mathcal{G}_{\varepsilon}|\leq2^{r^{1-\varepsilon}}$. This together with Lemma~\ref{lem:GTS} yields
\[
|\mathcal{Z}\setminus\mathcal{Y}_\varepsilon|\leq2^{\mathbf{c}(R)-\frac{r}{96}}\cdot|\mathcal{G}_{\varepsilon}|\leq2^{\mathbf{c}(R)-\frac{r}{96}+r^{1-\varepsilon}}.\qedhere
\]
\end{proof}

The following lemma is obtained by adapting the proof of~\cite[Theorem 3.4]{Morris-S2021}.

\begin{lemma}\label{lem:GlargeKlarge}
Let $R$ be a group of order $r$ that is neither abelian of exponent greater than $2$ nor generalized dicyclic, and let $\varepsilon\in (0,0.5)$. If $r$ is large enough, then
\[
|\mathcal{Y}_\varepsilon\setminus\mathcal{X}_\varepsilon|\leq2^{\mathbf{c}(R)-\frac{r}{96\log_2{r}}\log_2{e}+\log_2^2{r}+\log_2{r}-\log_2(96\log_2{r})}.
\]
\end{lemma}

\begin{proof}
Choose $r_\varepsilon\in\mathbb{N}$ such that $r^{0.5-\varepsilon}>\log_2^2{r}$ whenever $r\geq r_\varepsilon$, and assume $r\geq r_\varepsilon$ for the rest of the proof. Let $\mathcal{G}_\varepsilon=\{G(S):S\in\mathcal{Y}_\varepsilon\setminus\mathcal{X}_\varepsilon\}$. Since
\[
\mathcal{G}_{\varepsilon}\subseteq\{G\leq \Sym(R): R \textrm{ is maximal in }G,\,|G_1|>2^{r^{0.5-\varepsilon}},\,|\Core_G(R)|>96\log_2{r}\},
\]
one can read off from the proof of~\cite[Theorem 3.4]{Morris-S2021} that
\[
|\mathcal{G}_{\varepsilon}|\leq2^{\log_2^2{r}}\cdot\max\limits_{\substack{k\,|\,r\\k>96\log_2r}}\big(k^{r/k}(r/k)!\big).
\]
From the inequality $n!\leq n(n/e)^n$ we derive that, for $k>96\log_2{r}$,
\begin{align*}
\log_2\big(k^{r/k}(r/k)!\big)&\leq\frac{r}{k}\log_2{k}+\log_2\left(\frac{r}{k}\right)+\frac{r}{k}\log_2\left(\frac{r}{ek}\right)\\
&=\log_2\left(\frac{r}{k}\right)+\frac{r}{k}\log_2\left(k\cdot\frac{r}{ek}\right)\\
&=\log_2\left(\frac{r}{k}\right)+\frac{r}{k}(\log_2{r}-\log_2{e})\\
&<\log_2\left(\frac{r}{96\log_2r}\right)+\frac{r}{96\log_2r}(\log_2{r}-\log_2{e})\\
&=\log_2{r}-\log_2(96\log_2{r})+\frac{r}{96}-\frac{r}{96\log_2{r}}\log_2{e}.
\end{align*}
It follows that
\begin{align*}
|\mathcal{G}_{\varepsilon}|\leq2^{\log_2^2{r}+\log_2{r}-\log_2(96\log_2{r})+\frac{r}{96}-\frac{r}{96\log_2{r}}\log_2{e}}.
\end{align*}
Then by Lemma~\ref{lem:GTS}, we obtain
\begin{align*}
|\mathcal{Y}_\varepsilon\setminus\mathcal{X}_\varepsilon|&\leq2^{\mathbf{c}(R)-\frac{r}{96}}\cdot2^{\log_2^2r+\log_2r-\log_2(96\log_2r)+\frac{r}{96}-\frac{r}{96\log_2r}\log_2e}\\
&=2^{\mathbf{c}(R)-\frac{r}{96\log_2{r}}\log_2{e}+\log_2^2{r}+\log_2{r}-\log_2(96\log_2{r})}.\qedhere
\end{align*}
\end{proof}


Recall the notation introduced before Lemma~\ref{lem:Gsmall}. In particular,
\begin{equation*}
\begin{split}
\mathcal{W}=\mathcal{W}_{0.001}=\{S\in\mathcal{Z}:{}&|G(S)_1|>2^{r^{0.499}},\,|\Core_{G(S)}(R)|\leq96\log_2r,\\
&\text{every orbit of }\Core_{G(S)}(R)\text{ is stabilized by }G(S)_1\}.
\end{split}
\end{equation*}

\begin{lemma}\label{lem:WR}
Let $R$ be a finite group of order $r$ that is neither abelian of exponent greater than $2$ nor generalized dicyclic. For large enough $r$ we have $|\mathcal{W}|\leq2^{\mathbf{c}(R)-\frac{r}{18432\log_2{r}}+2\log_2^2{r}+3}$.
\end{lemma}

\begin{proof}
As $r$ is large enough, we have $r>96\log_2r$.
Let $\mathcal{N}$ be the set of nontrivial proper normal subgroups of $R$ of order at most $96\log_2{r}$. For $N\in\mathcal{N}$, we denote by $\mathcal{M}_N$ the set of inverse-closed subsets $S$ of $R$ such that there exists a non-identity $g\in\N_{\Aut(\Cay(R,S))}(N)$ that fixes the vertex $1$ and stabilizes each $N$-orbit. According to Lemma~\ref{lem:STEP2},
\begin{equation}\label{MN}
|\mathcal{M}_N|\leq2^{\mathbf{c}(R)-\frac{r}{192|N|}+\log_2^2{r}+3}.
\end{equation}

Take an arbitrary $S\in\mathcal{W}$. Let $G=G(S)$ and $K=\Core_{G(S)}(R)$. Since $G_1>1$ stabilizes every orbit of $K$ (because $S\in\mathcal{W}$), it follows that $K\neq1$. Moreover, it follows from $|R|=r>96\log_2r\geq|K|$ that $K\neq R$ and so $K\in\mathcal{N}$.
Since $K\trianglelefteq G$, we have $G\leq\N_{\Aut(\Cay(R,S))}(K)$. Consequently, any non-identity element of $G_1$ lies in $\N_{\Aut(\Cay(R,S))}(K)$ and stabilizes every orbit of $K$. In other words, $S\in\mathcal{M}_{K}$. This leads to
\begin{equation}\label{unionofMN}
\mathcal{W}\subseteq\bigcup_{N\in\mathcal{N}}\mathcal{M}_N.
\end{equation}

Since every subgroup of $R$ has a generating set of size at most $\log_2{r}$, we have at most $r^{\log_2{r}}$ choices for a subgroup $N$ of $R$. Combining this with~\eqref{MN} and~\eqref{unionofMN}, we conclude that
\begin{align*}
|\mathcal{W}|&\leq r^{\log_2{r}}\cdot\max_{N\in\mathcal{N}}\left\{2^{\mathbf{c}(R)-\frac{r}{192|N|}+\log_2^2{r}+3}\right\}\\
&\leq2^{\log_2^2{r}}\cdot2^{\mathbf{c}(R)-\frac{r}{192\cdot96\log_2{r}}+\log_2^2{r}+3}=2^{\mathbf{c}(R)-\frac{r}{18432\log_2{r}}+2\log_2^2{r}+3}.\qedhere
\end{align*}
\end{proof}

Combining Lemmas~\ref{lem:Gsmall}, \ref{lem:GlargeKlarge} and~\ref{lem:WR}, we obtain the following result.

\begin{proposition}\label{prop:ZtoT}
Let $R$ be a finite group of order $r$ that is neither abelian of exponent greater than $2$ nor generalized dicyclic. For large enough $r$ we have $|\mathcal{Z}\setminus\mathcal{T}|\leq2^{\mathbf{c}(R)-\frac{r}{18432\log_2{r}}+2\log_2^2{r}+5}$.
\end{proposition}

\begin{proof}
As $r$ is large enough, we see from Lemmas~\ref{lem:Gsmall}, \ref{lem:GlargeKlarge} and~\ref{lem:WR} that
\begin{align*}
|\mathcal{Z}\setminus\mathcal{Y}|\leq a_1\ &\text{ with }\ a_1=2^{\mathbf{c}(R)-\frac{r}{96}+r^{0.999}},\\
|\mathcal{Y}\setminus\mathcal{X}|\leq a_2\ &\text{ with }\ a_2=2^{\mathbf{c}(R)-\frac{r}{96\log_2{r}}\log_2{e}+\log_2^2{r}+\log_2{r}-\log_2(96\log_2{r})},\\
|\mathcal{W}|\leq a_3\ &\text{ with }\ a_3=2^{\mathbf{c}(R)-\frac{r}{18432\log_2{r}}+2\log_2^2{r}+3}.
\end{align*}
Observe that $a_3=\max\{a_1,a_2,a_3\}$ for large enough $r$. It follows that
\[
|\mathcal{Z}\setminus\mathcal{T}|=|(\mathcal{Z}\setminus\mathcal{Y})\cup(\mathcal{Y}\setminus\mathcal{X})\cup\mathcal{W}|
\leq|\mathcal{Z}\setminus\mathcal{Y}|+|\mathcal{Y}\setminus\mathcal{X}|+|\mathcal{W}|\leq3a_3<2^2a_3,
\]
as the proposition asserts.
\end{proof}

\section{Estimate $|\mathcal{T}|$}\label{sec:4}

By Proposition~\ref{prop:ZtoT} we have reduced the estimation of $|\mathcal{Z}|$ to that of $|\mathcal{T}|$. Recall that
\begin{equation*}
\begin{split}
\mathcal{T}=\{S\in\mathcal{Z}:{}&|G(S)_1|>2^{r^{0.499}},\,|\Core_{G(S)}(R)|\leq96\log_2r,\\
&\text{some orbit of }\Core_{G(S)}(R)\text{ is not stabilized by }G(S)_1\}.
\end{split}
\end{equation*}
For each $S\in\mathcal{T}$, let $K(S)=\Core_{G(S)}(R)$. It is known from the discussion in the previous section that $G(S)/K(S)$ is a primitive permutation group on $G(S)/R$ with point stabilizer $R/K(S)$ and a regular subgroup $G(S)_1K(S)/K(S)$.
For each $\mathcal{C}\in\{\textrm{HA, HS, HC, SD, CD, TW, AS, PA}\}$, let $\mathcal{T}_{\mathcal{C}}$ be the set of elements $S\in\mathcal{T}$ such that $G(S)/K(S)$ is of type $\mathcal{C}$ in its primitive action on the set of right cosets of $R$ in $G(S)$. Then
\begin{equation}\label{TR-division}
\mathcal{T}=\mathcal{T}_{\mathrm{HA}}\cup\mathcal{T}_{\mathrm{HS}}\cup\mathcal{T}_{\mathrm{HC}}\cup\mathcal{T}_{\mathrm{SD}}\cup\mathcal{T}_{\mathrm{CD}}\cup\mathcal{T}_{\mathrm{TW}}\cup\mathcal{T}_{\mathrm{AS}}\cup\mathcal{T}_{\mathrm{PA}}.
\end{equation}
Note that there are two actions of $G(S)$, namely the primitive action on the set $G(S)/R$ of right cosets of $R$ in $G(S)$ and the transitive action on the vertex set $R$ of the Cayley graph $\Cay(R,S)$. In the following, we will always emphasize the primitivity when referring to the first action, and the reader should keep in mind which action we are considering.

\begin{lemma}\label{lem:Q}
Suppose that $r$ is large enough. For each $S\in\mathcal{T}_{\mathrm{HA}}\cup\mathcal{T}_{\mathrm{SD}}\cup\mathcal{T}_{\mathrm{TW}}$, there exists a non-normal subgroup $Q$ of $R$ with $|Q|<r^{0.501}\log_2^3{r}$ such that the right $Q$-cosets in $R$ are precisely the orbits of some subgroup of $\Aut(\Cay(R,S))$.
\end{lemma}

\begin{proof}
Take an arbitrary $S\in\mathcal{T}_{\mathrm{HA}}\cup\mathcal{T}_{\mathrm{SD}}\cup\mathcal{T}_{\mathrm{TW}}$. Let $G=G(S)$, let $K=K(S)=\Core_G(R)$, and let $\overline{\phantom{u}}\colon G\to G/K$ be the quotient modulo $K$. Since $K$ is normal in $G$, the set $R/K$ of right $K$-cosets in $R$ forms a block system of $G$. Let $L$ be the kernel of the induced action of $G$ on $R/K$. Then the $L$-orbits are precisely the $K$-orbits, and hence some $L$-orbit is not stabilized by $G_1$ as $S\in\mathcal{T}$. Thereby we obtain from Lemma~\ref{lem:normalorbit} that $L\leq R$. It follows that $L=K$, and so $\overline{G}=G/K$ acts faithfully on $R/K$.

Let $N$ be the full preimage of $\Soc(\overline{G})$ under $\overline{\phantom{u}}$. Clearly, we have $K\leq N$.
Since $G$ is transitive on the vertex set of $\Cay(R,S)$ and $N$ is normal in $G$, the $N$-orbits on the vertices of $\Cay(R,S)$ form a block system of $G$.
Let $Q=1^N$ be the $N$-orbit containing the vertex $1$. Then $K=1^K\subseteq1^N=Q$. Since $R$ acts regularly on the vertex set of $\Cay(R,S)$ by right multiplication, we conclude that $Q$ is a subgroup $R$. Hence the block system formed by the $N$-orbits coincides with the set of right $Q$-cosets of $R$.

If $K=Q$, then every $N$-orbit has the form $x^N=Qx=Kx=xK=x^K$ and thus is a $K$-orbit. This would imply that $N$ stabilizes every $K$-orbit, that is, $\overline{N}$ is trivial, which is impossible as $\overline{N}=\Soc(\overline{G})$. Thus we have $K<Q$.
If $N\leq R$, then $\Soc(\overline{G})=\overline{N}\leq\overline{R}$, which is not possible as $\overline{R}$ is the point stabilizer of $\overline{G}$ acting primitively on $G/R$. Therefore, $N\nleq R$.

Suppose that $Q$ is normal in $R$. Since $R$ is maximal in $G$, we conclude that either $\N_{G}(Q)=R$ or $\N_{G}(Q)=G$. For $g\in N$ and $x\in R$, we have
\[
x^{gQ}=(x^g)^Q=x^gQ=Qx^g\ \text{ and }\ x^{Qg}=(x^{Q})^{g}=(xQ)^g=(Qx)^g.
\]
Recall that the $N$-orbits on $R$, which are the right $Q$-cosets in $R$, form a block system of $G$. Since $(Qx)^g$ is a block in this block system that contains $x^g$, it follows that $(Qx)^g=Qx^g$ and hence $x^{gQ}=x^{Qg}$. This means that $gQ=Qg$ for all $g\in N$, that is, $N\leq\N_G(Q)$. Since $N\nleq R$, it follows that  $\N_{G}(Q)\neq R$. This leads to $\N_{G}(Q)=G$, contradicting $K<Q$.

Thus $Q$ is non-normal in $R$. Denote the identity element of $R/K$ by $\overline{1}$. We have
\[
\overline{1}^{\overline{N}}=\overline{K^N}=\overline{(1^K)^N}=\overline{1^{KN}}=\overline{1^N}=\overline{Q}.
\]
This together with~\cite[Proposition 5.7]{Morris-S2021} gives $|\overline{Q}|\leq|R/K|^{0.501}\log_2^2{|R/K|}$, which yields
\[
|Q|=|K|\cdot|\overline{Q}|\leq |K|\cdot\left(\frac{r}{|K|}\right)^{0.501}\log_2^2{\left(\frac{r}{|K|}\right)}\leq r^{0.501}|K|^{0.499}\log_2^2{r}.
\]
Since $|K|\leq 96\log_2r$ and $r$ is large enough, it follows that
\[
|Q|\leq r^{0.501}(96\log_2r)^{0.499}\log_2^2{r}<r^{0.501}\log_2^3{r},
\]
completing the proof.
\end{proof}

\begin{lemma}\label{lem:HASDTW}
For large enough $r$, we have $|\mathcal{T}_{\mathrm{HA}}\cup\mathcal{T}_{\mathrm{SD}}\cup\mathcal{T}_{\mathrm{TW}}|\leq 2^{\mathbf{c}(R)-\frac{r^{0.499}}{8\log_2^3{r}}+\log_2^2{r}}$.
\end{lemma}

\begin{proof}
Let $\mathcal{Q}=\{Q\leq R:Q\ntrianglelefteq R,\,|Q|<r^{0.501}\log_2^3{r}\}$.
For each $Q\in\mathcal{Q}$, consider the set $\mathcal{N}_Q$ of inverse-closed subsets $S$ of $R$ such that the right $Q$-cosets in $R$ are precisely the orbits of some subgroup of $\Aut(\Cay(R,S))$. Then by Lemma~\ref{lem:Q} we have
\[
|\mathcal{T}_{\mathrm{HA}}\cup\mathcal{T}_{\mathrm{SD}}\cup\mathcal{T}_{\mathrm{TW}}|\leq\sum_{Q\in\mathcal{Q}}|\mathcal{N}_Q|.
\]

Take arbitrary $Q\in\mathcal{Q}$ and $S\in\mathcal{N}_Q$. Denote $\mathcal{D}=Q\backslash R/Q=\{QxQ: x\in R\}$. Let $\Delta\in\mathcal{D}$, and let $\Lambda$ be a right coset of $Q$ contained in $\Delta$. Recall by the definition of Cayley graphs that, for $x\in R$, the neighbourhood of $x$ in $\Cay(R,S)$ is $Sx$. For each $q\in Q$, Lemma~\ref{lem:orbit-equitable} asserts that $1$ and $q$ have the same number of neighbours in $\Lambda q$, which means
\[
|S\cap \Lambda q|=|Sq\cap \Lambda q|=|S\cap\Lambda|.
\]
This implies that for every right $Q$-coset contained in $\Delta$, say $\Lambda'$, we have $|S\cap\Lambda'|=|S\cap\Lambda|$. Since this $\Delta$ is chosen arbitrarily, the set $S$ must intersect $\Delta$ evenly for every $\Delta\in\mathcal{D}$. Hence by Proposition~\ref{prop:ns-nonnorm}, we obtain
\[
|\mathcal{N}_Q|\leq 2^{\mathbf{c}(R)-\frac{1}{8}|R\,{:}\,Q|}.
\]
Since $|Q|<r^{0.501}\log_2^3{r}$, it follows that
\[
|\mathcal{N}_Q|\leq2^{\mathbf{c}(R)-\frac{r^{0.499}}{8\log_2^3{r}}}.
\]
Note that $R$ has at most $2^{\log_2^2{r}}$ subgroups, which implies that $|\mathcal{Q}|\leq2^{\log_2^2{r}}$.
Therefore,
\[
|\mathcal{T}_{\mathrm{HA}}\cup\mathcal{T}_{\mathrm{SD}}\cup\mathcal{T}_{\mathrm{TW}}|\leq|\mathcal{Q}|\cdot2^{\mathbf{c}(R)-\frac{r^{0.499}}{8\log_2^3{r}}}=2^{\mathbf{c}(R)-\frac{r^{0.499}}{8\log_2^3{r}}+\log_2^2{r}}.\qedhere
\]
\end{proof}

Let $\Gamma$ be a graph on a vertex set $V$, and let $\mathcal{B}$ be a partition of $V$. If $\mathcal{B}$ is an equitable partition, that is, given any $B,C\in\mathcal{B}$, all vertices $v\in B$ have the same number of neighbours in $C$, then we denote this number by $e(B,C)$ for $(B,C)\in\mathcal{B}\times\mathcal{B}$. Let $\mathcal{B}$ be an equitable partition of $V$ into sets of equal size. Then $e(B,C)=e(C,B)$ for $B,C\in\mathcal{B}$. The \emph{odd quotient graph} $\Gamma^{\text{odd}}_{\mathcal{B}}$ of $\Gamma$ with respect to $\mathcal{B}$, defined in~\cite[Definition~6.1]{Morris-S2021}, is the graph whose vertices are the sets in $\mathcal{B}$, with an edge between $B$ and $C$ if and only if $e(B,C)$ is odd.
If $\mathcal{B}$ is the orbit partition of some $K\leq\Aut(\Gamma)$, then we call $\Gamma^{\text{odd}}_{\mathcal{B}}$ the odd quotient graph with respect to $K$.

\begin{lemma}\label{lem:HSHCASPA}
$|\mathcal{T}_{\mathrm{HS}}\cup\mathcal{T}_{\mathrm{HC}}\cup\mathcal{T}_{\mathrm{AS}}\cup\mathcal{T}_{\mathrm{PA}}|\leq 2^{\mathbf{c}(R)-\frac{r}{768\log_2r}+\log_2^2{r}+c_1}$ for some absolute constant $c_1$.
\end{lemma}

\begin{proof}
Let $\mathcal{K}=\{K(S): S\in\mathcal{T}_{\mathrm{HS}}\cup\mathcal{T}_{\mathrm{HC}}\cup\mathcal{T}_{\mathrm{AS}}\cup\mathcal{T}_{\mathrm{PA}}\}$. For $N\in\mathcal{K}$, denote by $\overline{\mathcal{T}}_N$ the set of inverse-closed subsets $X$ of $R/N$ such that there exists a subgroup $A$ of $\Aut(\Cay(R/N,X))$ with the following conditions:
\begin{itemize}
  \item $R/N$ is maximal and core-free in $A$;
  \item the vertex stabilizer in $A$ has size greater than $2^{|R/N|^{0.499}}$;
  \item the primitive action of $A$ by right multiplication on the set of right $(R/N)$-cosets in $A$ is of type $\mathrm{HS}$, $\mathrm{HC}$, $\mathrm{AS}$ or $\mathrm{PA}$.
\end{itemize}
Let $\mathcal{G}_N$ be the set of (labeled) Cayley graphs on $R/N$ with connection set in $\overline{\mathcal{T}}_N$.  For each $\Gamma\in\mathcal{G}_N$, let $\mathcal{N}_\Gamma$ be the set of $S$ such that the odd quotient graph of $\Cay(R,S)$ with respect to $N$ is $\Gamma$.

From the argument in~\cite[Pages~446~and~447]{Morris-S2021} we see that
for each $S\in\mathcal{T}_{\mathrm{HS}}\cup\mathcal{T}_{\mathrm{HC}}\cup\mathcal{T}_{\mathrm{AS}}\cup\mathcal{T}_{\mathrm{PA}}$, there exists (not necessarily nontrivial) $N$ in $\mathcal{K}$ such that the odd quotient graph of $\Cay(R,S)$ with respect to $N$ lies in $\mathcal{G}_N$.
This implies that
\[
\mathcal{T}_{\mathrm{HS}}\cup\mathcal{T}_{\mathrm{HC}}\cup\mathcal{T}_{\mathrm{AS}}\cup\mathcal{T}_{\mathrm{PA}}
\subseteq\bigcup_{N\in\mathcal{K}}\bigcup_{\ \Gamma\in\mathcal{G}_N}\mathcal{N}_\Gamma.
\]

Take an arbitrary $N\in\mathcal{K}$ and write $\overline{R}=R/N$. It is known from \cite[Theorems 5.2, 5.3 and 5.4]{Morris-S2021} that $|\overline{\mathcal{T}}_N|$ is bounded above by $2^{(3\cdot29!)!}+2^{(442^2)!^6\cdot6!}$. Since $|\mathcal{G}_N|=|\overline{\mathcal{T}}_N|$, it follows that
\begin{equation}\label{UnionOfN}
|\mathcal{T}_{\mathrm{HS}}\cup\mathcal{T}_{\mathrm{HC}}\cup\mathcal{T}_{\mathrm{AS}}\cup\mathcal{T}_{\mathrm{PA}}|
\leq|\mathcal{K}|\cdot\big(2^{(3\cdot29!)!}+2^{(442^2)!^6\cdot6!}\big)\cdot\max\limits_{N\in\mathcal{K}}\max\limits_{\ \Gamma\in\mathcal{G}_N}|\mathcal{N}_\Gamma|.
\end{equation}
Now take $\Gamma\in\mathcal{G}_N$ and consider $|\mathcal{N}_{\Gamma}|$, namely, the number of choices for $S$ such that the odd quotient graph of $\Cay(R,S)$ with respect to $N$ is $\Gamma$. Let $S$ be such a choice. For each $x\in R$,
\[
(S\cap xN)^{-1}=S^{-1}\cap Nx^{-1}=S\cap x^{-1}N,
\]
which means that $S\cap x^{-1}N$ is determined by $S\cap xN$ when $xN\notin\mathcal{I}(\overline{R})$. According to the definition of odd quotient graph, this subset $S$ of $R$ must satisfy the following conditions:
\begin{itemize}
  \item for each $xN\in\mathcal{I}(\overline{R})$, the set $S\cap xN$ is an inverse-closed subset of $xN$;
  \item for each pair \{$xN,\,x^{-1}N\}\subseteq\overline{R}\setminus\mathcal{I}(\overline{R})$, the set $S\cap xN$ is a subset of $xN$ such that $|S\cap xN|$ is odd if and only if $xN$ is adjacent to $N$ in $\Gamma$, and the other set $S\cap x^{-1}N$ is determined by $S\cap xN$.
\end{itemize}
Since the odd quotient graph $\Gamma$ is given, the parity of $|S\cap xN|$ is determined for each $xN\in\overline{R}$.
Note from~\eqref{eq:parity} that the number of subsets of $xN$ with a given parity is $2^{|xN|-1}$. Hence the number of choices for $S$ is at most
\begin{align*}
&\prod_{xN\in\mathcal{I}(\overline{R})}2^{\mathbf{c}(xN)}\prod_{\{xN,\,x^{-1}N\}\subseteq\overline{R}\setminus\mathcal{I}(\overline{R})}2^{|xN|-1}\\
={}&\left(\prod_{xN\in\mathcal{I}(\overline{R})}2^{\mathbf{c}(xN)}\prod_{\{xN,\,x^{-1}N\}\subseteq\overline{R}\setminus\mathcal{I}(\overline{R})}2^{\frac{|xN|
+|x^{-1}N|}{2}}\right)\cdot2^{-\frac{1}{2}|\overline{R}\setminus\mathcal{I}(\overline{R})|}\\
={}&2^{\mathbf{c}(R)-\frac{1}{2}|\overline{R}\setminus\mathcal{I}(\overline{R})|},
\end{align*}
which means $|\mathcal{N}_{\Gamma}|\leq2^{\mathbf{c}(R)-\frac{1}{2}|\overline{R}\setminus\mathcal{I}(\overline{R})|}$.
It is known from Lemma~\ref{lem:typeex} that $R/N$ is not an elementary abelian $2$-group. Hence by Lemma~\ref{lem:inv-ea2} we have $|\overline{R}\setminus\mathcal{I}(\overline{R})|\geq|\overline{R}|/4$,
and so
\[
|\mathcal{N}_{\Gamma}|\leq2^{\mathbf{c}(R)-\frac{1}{8}|\overline{R}|}.
\]
Since $N\in\mathcal{K}$, we have $|N|\leq96\log_2r$ and so $|\overline{R}|=r/|N|\geq r/96\log_2r$. This leads to
\[
|\mathcal{N}_{\Gamma}|\leq2^{\mathbf{c}(R)-\frac{r}{768\log_2r}},
\]
which together with~\eqref{UnionOfN} and the observation $|\mathcal{K}|\leq 2^{\log_2^2r}$ completes the proof.
\end{proof}

\begin{lemma}\label{lem:CD}
$|\mathcal{T}_{\mathrm{CD}}|\leq 2^{\mathbf{c}(R)-\frac{r}{9216\log_2r}+\log_2^4{r}+\frac{1}{5}\log_2^3{r}+c_2\log_2^2{r}+\log_2{r}}$ for some absolute constant $c_2$.
\end{lemma}

\begin{proof}
Let $\mathcal{K}=\{K(S): S\in \mathcal{T}_{\mathrm{CD}}\}$. Take an arbitrary $N\in\mathcal{K}$ and write $n=|N|$. Denote
\begin{align*}
\mathcal{M}_N&=\{S\in\mathcal{T}_{\mathrm{CD}}: K(S)=N\},\\
\mathcal{F}_N&=\{G(S)/N:S\in\mathcal{M}_N\}.
\end{align*}
For each $Y\in\mathcal{F}_N$, let $\mathcal{N}_{Y}=\{S\in\mathcal{M}_N:G(S)/N=Y\}$. It is clear that
\[
\mathcal{M}_N=\bigcup_{Y\in\mathcal{F}_N}\mathcal{N}_{Y},
\]
which yields that
\begin{equation}\label{unionofcd}
|\mathcal{M}_N|\leq|\mathcal{F}_N|\cdot\max_{Y\in\mathcal{F}_N}|\mathcal{N}_{Y}|.
\end{equation}
Note that each $G(S)/N\in\mathcal{F}_N$ acts faithfully on $G(S)/R$ as a primitive group of type $\mathcal{\mathrm{CD}}$ with a regular subgroup $G(S)_1N/N\cong G(S)_1$ of order $2^{r^{0.499}}\geq2^{|R/N|^{0.499}}$. By the proof of \cite[Theorem 5.10]{Morris-S2021}, we have
\begin{equation}\label{GK}
|\mathcal{F}_N|\leq2^{\log_2^4{\frac{r}{n}}+\frac{1}{5}\log_2^3{\frac{r}{n}}+c_3\log_2^2{\frac{r}{n}}+\log_2{\frac{r}{n}}}
\end{equation}
for some absolute constant $c_3$.

Take arbitrary $Y\in\mathcal{F}_N$ and $S\in\mathcal{N}_Y$, so that $Y=G(S)/N$. Let $H=G(S)_1N/N$. For an $H$-orbit $\Delta$ on $R/N$, let $\widetilde{\Delta}$ be the union of the right $N$-cosets in $\Delta$.
Let $Nx\in\Delta$ and $g\in G(S)_1$. Note that $S$ is exactly the neighborhood of the vertex $1$, and $g$ is an automorphism of $\Cay(R,S)$ that fixes $1$. It follows that $S^g=S$ and so
\[
|S\cap Nx|=|(S\cap Nx)^g|=|S^g\cap(Nx)^g|=|S\cap(Nx)^g|.
\]
This means that $S$ intersects $\widetilde{\Delta}$ evenly for every $H$-orbit $\Delta$ on $R/N$.
Since Lemma~\ref{lem:typeex} implies that $R/N$ is neither abelian nor generalized dicyclic, it follows from Proposition~\ref{prop:ns-G1} that
\[
|\mathcal{N}_{Y}|\leq 2^{\mathbf{c}(R)-\frac{1}{96}\cdot\frac{r}{n}}
\]
for all $Y\in\mathcal{F}_N$. Thus, by~\eqref{unionofcd} and~\eqref{GK}, we obtain
\[
|\mathcal{M}_N|\leq2^{\log_2^4{\frac{r}{n}}+\frac{1}{5}\log_2^3{\frac{r}{n}}+c_3\log_2^2{\frac{r}{n}}+\log_2{\frac{r}{n}}}\cdot2^{\mathbf{c}(R)-
\frac{r}{96n}}=2^{\mathbf{c}(R)-\frac{r}{96n}+\log_2^4{\frac{r}{n}}+\frac{1}{5}\log_2^3{\frac{r}{n}}+c_3\log_2^2{\frac{r}{n}}+\log_2{\frac{r}{n}}}.
\]
Moreover, $n=|N|\leq 96\log_2r$ as $N\in\mathcal{K}$. Hence
\[
|\mathcal{M}_N|\leq2^{\mathbf{c}(R)-\frac{r}{96\cdot96\log_2r}+\log_2^4{r}+\frac{1}{5}\log_2^3{r}+c_3\log_2^2{r}+\log_2{r}}.
\]
Since $R$ has at most $2^{\log_2^2{r}}$ subgroups, there are at most $2^{\log_2^2{r}}$ choices for $N$. Consequently,
\begin{align*}
|\mathcal{T}_{\mathrm{CD}}|&=2^{\log_2^2{r}}\cdot\max_{N\in\mathcal{K}}|\mathcal{M}_N|\\
&\leq2^{\log_2^2{r}}\cdot 2^{\mathbf{c}(R)-\frac{r}{96\cdot96\log_2r}+\log_2^4{r}+\frac{1}{5}\log_2^3{r}+c_3\log_2^2{r}+\log_2{r}}\\
&=2^{\mathbf{c}(R)-\frac{r}{9216\log_2r}+\log_2^4{r}+\frac{1}{5}\log_2^3{r}+c_2\log_2^2{r}+\log_2{r}},
\end{align*}
where $c_2=c_3+1$ is an absolute constant.
\end{proof}

Comparing the three upper bounds in Lemmas~\ref{lem:HASDTW}, \ref{lem:HSHCASPA} and~\ref{lem:CD}, it is clear that for large enough $r$, the bound $2^{\mathbf{c}(R)-\frac{r^{0.499}}{8\log_2^3{r}}+\log_2^2{r}}$ in Lemma~\ref{lem:HASDTW} is the largest. Thereby, in view of~\eqref{TR-division}, we obtain the following result.

\begin{proposition}\label{prop:totalT}
Let $R$ be a finite group of order $r$ that is neither abelian of exponent greater than $2$ nor generalized dicyclic. For large enough $r$ we have $|\mathcal{T}|\leq2^{\mathbf{c}(R)-\frac{r^{0.499}}{8\log_2^3{r}}+\log_2^2{r}+2}$.
\end{proposition}

\section{Proof of main results}

We are now ready to prove Theorem~\ref{Thm:main} and its corollaries.

\begin{proof}[Proof of Theorem~$\ref{Thm:main}$]
For the group $R$ that is neither abelian of exponent greater than $2$ nor generalized dicyclic, recall from the beginning of Section~\ref{Sec4} that
\[
\mathcal{Z}=\{S\subseteq R: S^{-1}=S,\,\Aut(\Cay(R,S))>R\}=(\mathcal{Z}\setminus\mathcal{T})\cup\mathcal{T}.
\]
Let $c_4$ be a constant number such that $r\geq c_4$ is large enough for Propositions~\ref{prop:ZtoT} and~\ref{prop:totalT} to hold. Comparing the two upper bounds in Propositions~\ref{prop:ZtoT} and~\ref{prop:totalT}, it is clear that the bound
\begin{equation}\label{boundforT}
2^{\mathbf{c}(R)-\frac{r^{0.499}}{8\log_2^3{r}}+\log_2^2{r}+2}
\end{equation}
in Proposition~\ref{prop:totalT} is the larger one when $r$ is sufficiently large, that is, whenever $r\geq c_5$ for some constant number $c_5$. Let $c_0=\max\{c_4,c_5\}$. It follows that, whenever $r\geq c_0$, we derive an upper bound for $|\mathcal{Z}|$ by doubling~\eqref{boundforT}, so that
\[
|\mathcal{Z}|\leq 2^{\mathbf{c}(R)-\frac{r^{0.499}}{8\log_2^3{r}}+\log_2^2{r}+3}.
\]
Recall that the number of inverse-closed subsets of $R$ is $2^{\mathbf{c}(R)}$. Therefore,
\[
P(R)=1-\frac{|\mathcal{Z}|}{2^{\mathbf{c}(R)}}\geq 1-2^{-\frac{r^{0.499}}{8\log_2^3{r}}+\log_2^2{r}+3}
\]
for $r\geq c_0$. This completes the proof.
\end{proof}

\begin{proof}[Proof of Corollary~$\ref{unlab}$]
Let $P(R)$ be the proportion of inverse-closed subsets $S$ of $R$ such that $\Cay(R,S)$ is a GRR of $R$ as in Theorem~$\ref{Thm:main}$. Recall that $\CG(R)$ is the set of Cayley graphs over $R$ up to isomorphism and that $\GRR(R)$ is the set of GRRs over $R$ up to isomorphism.

Let $\mathcal{S}(R)$ be the set of inverse-closed subsets of $R$. Since $|\mathcal{S}(R)|(1-P(R))$ counts the number of labeled Cayley graphs of $R$ that are not GRRs, it is clear that
\[
|\CG(R)|-|\GRR(R)|\leq |\mathcal{S}(R)|(1-P(R)).
\]
Similar to the proof of~\cite[Theorem 1.5]{Morris-S2021}, we have
\[
|\GRR(R)|\geq\frac{|\mathcal{S}(R)|P(R)}{|\Aut(R)|}.
\]
If follows that
\[
\frac{|\GRR(R)|}{|\CG(R)|}=\left(1+\frac{|\CG(R)|-|\GRR(R)|}{|\GRR(R)|}\right)^{-1}\geq\left(1+\frac{|\mathcal{S}(R)|(1-P(R))}{|\mathcal{S}(R)|P(R)/|\Aut(R)|}\right)^{-1}.
\]
Note that every $\alpha\in\Aut(R)$ is determined by the images of elements in a generating set for $R$ under $\alpha$ and that $R$ has a generating set of size at most $\log_2{r}$. We have $|\Aut(R)|\leq 2^{a(r)}$ with $a(r)=\log_2^2{r}$, and so
\[
\frac{|\GRR(R)|}{|\CG(R)|}\geq\left(1+\frac{|\mathcal{S}(R)|(1-P(R))}{|\mathcal{S}(R)|P(R)/2^{a(r)}}\right)^{-1}
=\frac{1}{1-2^{a(r)}+2^{a(r)}/P(R)}.
\]
Moreover, Theorem~\ref{Thm:main} states that $P(R)\geq1-2^{-b(r)}$ for $r\geq c_0$, where $b(r)=\frac{r^{0.499}}{8\log_2^3{r}}-\log_2^2{r}-3$ while $c_0$ is a constant number. Thus, whenever $r\geq c_0$,
\begin{align*}
\frac{|\GRR(R)|}{|\CG(R)|}&\geq\frac{1}{1-2^{a(r)}+2^{a(r)}/\big(1-2^{-b(r)}\big)}\\
&=\frac{1-2^{-b(r)}}{1-2^{-b(r)}+2^{a(r)-b(r)}}=1-\frac{2^{a(r)-b(r)}}{1-2^{-b(r)}+2^{a(r)-b(r)}}>1-2^{a(r)-b(r)}.
\end{align*}
Since $a(r)-b(r)=-\frac{r^{0.499}}{8\log_2^3{r}}+2\log_2^2{r}+3$, the proof is complete.
\end{proof}

\begin{proof}[Proof of Corollary~$\ref{normality}$]

First assume that $R$ is an abelian group of exponent greater than $2$. From~\cite[Theorem~1.7]{Dobson2016} we deduce that the number of inverse-closed subsets $S$ of $R$ with $|\Aut(\Cay(R,S))|=2|R|$ is at least $2^{\mathbf{c}(R)}-2^{\frac{|\mathcal{I}(R)|}{2}+\frac{11r}{24}+2\log_2^2{r}+2}$, where recall that $\mathbf{c}(R)$ and $|\mathcal{I}(R)|$ satisfy~\eqref{eq:cr}. Therefore, the proportion of inverse-closed subsets $S$ of $R$ such that $\Cay(R,S)$ is a normal Cayley graph of $R$ is at least
\begin{equation}\label{abeliannormal}
\frac{2^{\mathbf{c}(R)}-2^{\frac{|\mathcal{I}(R)|}{2}+\frac{11r}{24}+2\log_2^2{r}+2}}{2^{\mathbf{c}(R)}}=1-2^{-\frac{r}{24}+2\log_2^2{r}+2}.
\end{equation}

Next assume that $R$ is a generalized dicyclic group. Since $R\ncong \Quat_8\times \Cy_2^m$ for any nonnegative integer $m$, we see from~~\cite[Theorem~1.5]{Morris-SV2015} that the number of inverse-closed subsets $S$ of $R$ with $|\Aut(\Cay(R,S))|=2|R|$ is at least $2^{\mathbf{c}(R)}-2^{\frac{|\mathcal{I}(R)|}{2}+\frac{23r}{48}+2\log_2^2{r}+5}$. Hence the proportion of inverse-closed subsets $S$ of $R$ such that $\Cay(R,S)$ is a normal Cayley graph of $R$ is at least
\begin{equation}\label{gedycnormal}
\frac{2^{\mathbf{c}(R)}-2^{\frac{|\mathcal{I}(R)|}{2}+\frac{23r}{48}+2\log_2^2{r}+5}}{2^{\mathbf{c}(R)}}=1-2^{-\frac{r}{48}+2\log_2^2{r}+5}.
\end{equation}

Now assume that $R$ is neither abelian of exponent greater than $2$ nor generalized dicyclic. As $r$ is large enough, it follows from Theorem~$\ref{Thm:main}$ that the proportion of inverse-closed subsets $S$ of $R$ with $\Cay(R,S)$ normal is at least
\begin{equation}\label{mainnormal}
1-2^{-\frac{r^{0.499}}{8\log_2^3{r}}+\log_2^2{r}+3}.
\end{equation}

Comparing the three lower bounds \eqref{abeliannormal}, \eqref{gedycnormal} and~\eqref{mainnormal} for large enough $r$, we see that the last one is the smallest. Accordingly, the proportion of inverse-closed subsets $S$ of $R$ such that $\Cay(R,S)$ is a normal Cayley graph of $R$ is at least $1-2^{-\frac{r^{0.499}}{8\log_2^3{r}}+\log_2^2{r}+3}$ for all $R$ such that $R\ncong \Quat_8\times \Cy_2^m$ for any nonnegative integer $m$.
\end{proof}

\section*{Acknowledgments}
The second author was supported by the Melbourne Research Scholarship. The authors are grateful to the anonymous referee for valuable comments and suggestions, which have helped improve the paper.


\end{document}